\documentclass[12pt, a4paper, headings=standardclasses, openany, bibliography=totocnumbered]{amsart}

\usepackage[USenglish]{babel}
\usepackage{amsmath, amssymb, amsfonts, amsthm, mathtools, bbm}
\usepackage[style=alphabetic,backend=biber]{biblatex}
\usepackage{hyperref}
\usepackage{cleveref}
\usepackage{tikz-cd}
\usepackage{faktor}
\usepackage{enumitem}
\usepackage{tikz}
\usetikzlibrary{matrix,decorations.pathreplacing}
\usetikzlibrary{shapes.geometric,calc}
\usetikzlibrary{arrows.meta,decorations.markings}
\usetikzlibrary{positioning}

\tikzset{
	dot/.style={circle, fill=black, inner sep=0pt, outer sep=0pt, minimum size=3pt}
}

\newcommand{\seq}[1]{%
	\begin{tikzpicture}[baseline=-0.6ex]
		\foreach \x [count=\i] in {#1} {
			\ifnum\x=1
			\node[dot] at (0.25*\i,0) {};
			\else
			\node[dot] at (0.25*\i,0) {};
			\node[dot] at (0.25*\i,0.18) {};
			\fi
		}
	\end{tikzpicture}%
}

\usepackage{ytableau} %for Young diagrams

\ytableausetup{centertableaux, boxsize=2em}

\AtBeginBibliography{\footnotesize} 

\newtheorem{theo}{Theorem}
\numberwithin{theo}{section}
\newtheorem{lemm}[theo]{Lemma}

\newtheorem{prop}[theo]{Proposition}
\newtheorem{koro}[theo]{Corollary}
\newtheorem*{thmA}{Theorem B}
\newtheorem*{thmB}{Theorem D}

\theoremstyle{definition}
\newtheorem{defi}[theo]{Definition}
\newtheorem{beis}[theo]{Example}
\newtheorem*{defA}{Definition A}
\newtheorem*{beisA}{Example C}
\newtheorem{rema}[theo]{Remark}
\newtheorem*{rema*}{Remark}

\newtheorem{warn}[theo]{Warning}

\newtheorem{nota}[theo]{Notation}

\DeclareMathOperator{\Sym}{Sym} %symmetric group
\newcommand{\Rep}{\operatorname{Rep}_k} %category of representations
\newcommand{\blank}{{-}} %blank
\def\acts{\curvearrowright} %group action
\newcommand{\N}{\mathbb{N}} %natural numbers
\newcommand{\mN}{\mathbb{N}_{0}} %natural numbers with zero
\newcommand{\mQ}{\mathbb{Q}} %rational numbers
\newcommand{\cC}{{\mathcal{S}}} %category C
\newcommand{\cF}{{\mathcal{F}}} %filtration F
\newcommand{\cS}{{\mathbb{X}}} %graph S
\newcommand{\cs}{{\mathsf{x}}} %finite graph S
\newcommand{\ca}{\mathsf{a}} %amalgamation a
\DeclareMathOperator{\Hom}{Hom}
\DeclareMathOperator{\End}{End}
\newcommand{\id}{\mathrm{id}}
\newcommand{\cO}{{\mathcal{O}}} %category O or orbits
\DeclareMathOperator{\age}{age} %age
\newcommand{\am}{\cs_1\cup_{a, \cs_0}\cs_2} %amalgamation
\DeclareMathOperator{\Aut}{Aut} %automorphism group
\newcommand{\gr}{\age(\cS)} %class of graphs
\renewcommand{\sl}{\mathfrak{sl}} %special linear Lie algebra
\DeclareMathOperator{\Maps}{Maps} %Maps
\newcommand{\spann}{\mathrm{span}} %linear span
\DeclareMathOperator{\pt}{pt} %one point
\newcommand{\cbox}[1]{ \vcenter{ \hbox{#1} } }
\DeclareMathOperator{\M}{M} %matrices
\newcommand{\C}{\mathbb{C}}
\DeclareMathOperator{\pa}{Part}
\newcommand{\gl}{\mathfrak{gl}}
\DeclareMathOperator{\GL}{GL}

\newcommand{\Perm}{\underline{\operatorname{Perm}}(G; \mu)}
\renewcommand{\H}{\mathsf{H}}
\newcommand\restr[2]{{% we make the whole thing an ordinary symbol
		\left.\kern-\nulldelimiterspace % automatically resize the bar with \right
		#1 % the function
		\vphantom{\big|} % pretend it's a little taller at normal size
		\right|_{#2} % this is the delimiter
}} %restriction of functions
\renewcommand{\acts}{\circlearrowright}

\begin{document}
	\title[Lie algebra actions for oligomorphic groups]{Infinite sequences via Lie algebra actions for oligomorphic groups}
	%\author{Zbigniew Wojciechowski}
	\author[Z. Wojciechowski]{Zbigniew Wojciechowski}
	\address{Institut für Algebra, Technische Universität Dresden, Germany}
	\email{zbigniew.wojciechowski@tu-dresden.de}
	\maketitle
	\begin{abstract}
		Many integer sequences arise as numbers of $G$-orbits on $\binom{X}{n}$ as $n$ varies, for a permutation group $G\subseteq \Sym(X)$. For finite $X$, Stanley proved that these finite sequences increase towards the middle using an action of the Lie algebra $\sl_2(\C)$.
		For infinite sets $X$, and hence infinite sequences, Cameron provided an argument for monotonicity by identifying orbits with a vector space basis of the orbit algebra $\H_{G,X}^{\star}$, and proving injectivity of a certain operator $\H_{G,X}^{\star}\to \H_{G,X}^{\star+1}$. In this paper we generalize Stanley's approach to oligomorphic groups by extending Cameron's operator to a full $\sl_2(\C)$-action on $\H_{G,X}^{\star}$. As intermediate step, we define for every oligomorphic permutation group $G\subseteq \Sym(X)$ the \mbox{$X$-th} tensor power $(k^r)^{\otimes X}$, generalizing work of Entova-Aizenbud. We show that this space carries natural commuting actions of $G$ and the Lie algebra $\gl_r(k)$, the latter depending on a Harman--Snowden measure $\mu$ on $G$. 
		We show that $\H_{G,X}^{\star}\subseteq (\C^2)^{\otimes X}$ has a filtration by lowest weight Verma modules. We explain how our approach applies to Fibonacci numbers, Tribonacci numbers, etc.\ by constructing measures on products with $(\mQ,<)$. 
	\end{abstract}

\begin{figure}[ht]
	\centering
	
	\begin{tikzpicture}[
		element/.style={inner sep=2pt, outer sep=2pt},
		>=Stealth,
		arrow/.style={red!60!black, ->, double=white, double distance=1.8pt, line width=0.6pt},yscale=1.1
		]
		
		% Shift everything to start at x=0.5cm from left edge
		\begin{scope}[xshift=0.5cm]
			
			% Column spacing: 2.3cm
			% Increased vertical spacing for better number placement
			
			% Column 0: n=0 at x=0
			\node[element] (n0) at (0,0) {$\emptyset$};
			
			% Column 1: n=1 at x=2.3
			\node[element] (n1-1) at (2.3,0) {\seq{1}};
			
			% Column 2: n=2 at x=4.6 (increased vertical spacing: 1.5 instead of 1.2)
			\node[element] (n2-1) at (4.6,1.5) {\seq{2}};
			\node[element] (n2-2) at (4.6,-1.5) {\seq{1,1}};
			
			% Column 3: n=3 at x=6.9 (increased vertical spacing: 2.3 instead of 1.8)
			\node[element] (n3-1) at (6.9,2.3) {\seq{2,1}};
			\node[element] (n3-2) at (6.9,0) {\seq{1,2}};
			\node[element] (n3-3) at (6.9,-2.3) {\seq{1,1,1}};
			
			% Column 4: n=4 at x=9.2 (increased vertical spacing: 3.0 instead of 2.4)
			\node[element] (n4-1) at (9.2,3.0) {\seq{2,2}};
			\node[element] (n4-2) at (9.2,1.5) {\seq{2,1,1}};
			\node[element] (n4-3) at (9.2,0) {\seq{1,2,1}};
			\node[element] (n4-4) at (9.2,-1.5) {\seq{1,1,2}};
			\node[element] (n4-5) at (9.2,-3.0) {\seq{1,1,1,1}};
			
			% Draw all edges from left (smaller n) to right (larger n)
			
			% From n=0 to n=1
			\draw[arrow] (n0) -- (n1-1);
			
			% From n=1 to n=2
			\draw[arrow] (n1-1) -- (n2-1);
			\draw[arrow] (n1-1) -- (n2-2);
			
			% From n=2 to n=3
			\draw[arrow] (n2-1) -- (n3-1);
			\draw[arrow] (n2-1) -- (n3-2);
			\draw[arrow] (n2-2) -- (n3-1);
			\draw[arrow] (n2-2) -- (n3-2);
			\draw[arrow] (n2-2) -- (n3-3);
			
			% From n=3 to n=4
			% Adding a 1 to [2,1]
			\draw[arrow] (n3-1) -- (n4-2);
			\draw[arrow] (n3-1) -- (n4-3);
			% Changing 1 to 2 in [2,1]
			\draw[arrow] (n3-1) -- (n4-1);
			
			% Adding a 1 to [1,2]
			\draw[arrow] (n3-2) -- (n4-3);
			\draw[arrow] (n3-2) -- (n4-4);
			% Changing 1 to 2 in [1,2]
			\draw[arrow] (n3-2) -- (n4-1);
			
			% Adding a 1 to [1,1,1]
			\draw[arrow] (n3-3) -- (n4-5);
			% Changing a 1 to 2 in [1,1,1]
			\draw[arrow] (n3-3) -- (n4-2);
			\draw[arrow] (n3-3) -- (n4-3);
			\draw[arrow] (n3-3) -- (n4-4);
			
			% Add numbers placed directly on top of each arrow with white background
			% n=0 -> n=1
			\node[fill=white, inner sep=1pt, font=\small] at ($(n0)!0.5!(n1-1)$) {1};
			
			% n=1 -> n=2
			\node[fill=white, inner sep=1pt, font=\small] at ($(n1-1)!0.5!(n2-1)$) {2};
			\node[fill=white, inner sep=1pt, font=\small] at ($(n1-1)!0.5!(n2-2)$) {2};
			
			% n=2 -> n=3
			\node[fill=white, inner sep=1pt, font=\small] at ($(n2-1)!0.5!(n3-1)$) {1};
			\node[fill=white, inner sep=1pt, font=\small] at ($(n2-1)!0.4!(n3-2)$) {1};
			\node[fill=white, inner sep=1pt, font=\small] at ($(n2-2)!0.4!(n3-1)$) {2};
			\node[fill=white, inner sep=1pt, font=\small] at ($(n2-2)!0.5!(n3-2)$) {2};
			\node[fill=white, inner sep=1pt, font=\small] at ($(n2-2)!0.5!(n3-3)$) {3};
			
			% n=3 -> n=4
			% [2,1] edges
			\node[fill=white, inner sep=1pt, font=\small] at ($(n3-1)!0.4!(n4-2)$) {2};
			\node[fill=white, inner sep=1pt, font=\small] at ($(n3-1)!0.6!(n4-3)$) {1};
			\node[fill=white, inner sep=1pt, font=\small] at ($(n3-1)!0.5!(n4-1)$) {2};
			
			% [1,2] edges
			\node[fill=white, inner sep=1pt, font=\small] at ($(n3-2)!0.4!(n4-3)$) {1};
			\node[fill=white, inner sep=1pt, font=\small] at ($(n3-2)!0.3!(n4-4)$) {2};
			\node[fill=white, inner sep=1pt, font=\small] at ($(n3-2)!0.3!(n4-1)$) {2};
			
			% [1,1,1] edges
			\node[fill=white, inner sep=1pt, font=\small] at ($(n3-3)!0.5!(n4-5)$) {4};
			\node[fill=white, inner sep=1pt, font=\small] at ($(n3-3)!0.3!(n4-2)$) {2};
			\node[fill=white, inner sep=1pt, font=\small] at ($(n3-3)!0.5!(n4-3)$) {2};
			\node[fill=white, inner sep=1pt, font=\small] at ($(n3-3)!0.5!(n4-4)$) {2};
		\end{scope}
	\end{tikzpicture}
	
	\caption{Cameron's operator $e$ on $\H_{G,X}^{\star}$ for Fibonacci}
	\label{the action of e}
\end{figure}

	\section{Introduction}
	\subsection{Motivation}
		 Many counting problems in algebraic combinatorics can be translated to counting the number $G\backslash 2^X$ of $G$-orbits on the power set of a finite $G$-set $X$ for some group $G$. Famous problems of this form involve counting the number of necklaces with $n$ beads of $2$ colors, the number of partitions fitting into a $m\times n$ rectangle, and the number of graphs with $n$ vertices (see \cite{reinernote} for a great survey). There one studies the rank generating function
		 $r_X(q)=\sum_{n=0}^{|X|} a_n q^n$ where $a_n=\left|G\backslash \binom{X}{n}\right|$. The coefficients $a_n$ of $r_X(q)$ are \textit{symmetric}, i.e.\ satisfy $a_n=a_{|X|-n}$, and \textit{unimodal}, i.e.\ increasing towards the middle $a_0\leq a_1\leq \cdots \leq a_{\left\lfloor\! \frac{|X|}{2}\!\right\rfloor}$, by a result of Stanley. His original argument in \cite{stanley80b,stanley80a} works as follows. He considers the action of the Lie algebra $\mathfrak{sl}_2(\C)$ on $\left(\C^{2}\right)^{\otimes |X|}$, which commutes with the action of $G$. He then considers the subspace of $G$-invariant vectors, which has a natural basis labeled by $G$-orbits. Since this subspace is stable under the $\sl_2(\C)$-action, one may decompose it into irreducible summands corresponding to finite sequences of the form $1,1,1,\ldots, 1$, which provides unimodality. This paper is concerned with permutation groups $G\subseteq \Sym(X)$ of possibly infinite sets $X$ for which $\binom{X}{n}$ admits finitely many $G$-orbits for all $n\in \mN$.
		 %, i.e.\ for which the formal power series $r_X(q)=\sum_{n\geq 0}a_n q^n\in \mN[[q]]$ is well-defined. 
		 Such $G\subseteq \Sym(X)$ is called oligomorphic. Many infinite integer sequences arise in this way, for instance, Fibonacci numbers, Catalan numbers, and numbers of partitions (see \cite{cameron00}).
		 Cameron shows in \cite[\S5]{cameron90} that provided $X$ is infinite the sequence $a_0,a_1,a_2,\ldots$ is monotonically increasing.
		 He first considers $\H_{G,X}^{\star}=\bigoplus_{n\in \mN}\Maps\left(\binom{X}{n}, k\right)^{G}$, which satisfies $\dim_k \H_{G,X}^n=\left|G\backslash \binom{X}{n}\right|$.
		 This space is naturally a commutative $k$-algebra via the intersection product. Cameron shows that the operator $\H_{G,X}^{\star}\to \H_{G,X}^{\star+1}$, which multiplies with the constant $1$-function from $\H_{G,X}^1$, is injective. This implies monotonicity of the sequence.
		 The goal of this paper is to put Stanley's and Cameron's approaches into a uniform framework by defining Lie algebra actions on orbit algebras for oligomorphic groups.
		 
		 \subsection{Results}
		 Let $k$ be a field. Let $G\subseteq \Sym(X)$ be an oligomorphic permutation group. We equip $G$ with the topology induced by the discrete topology on $X$. The following are the key definitions.
		 \begin{defA}[Definition \ref{definition par} and \ref{definition tensor power}]
		 	Consider the $G$-set
		 	\[
		 	\pa_r(X)\coloneqq  \left\{ (X_1,\cdots, X_r) \mid \substack{\exists U\subseteq G \text{ open with } U\cdot X_i\subseteq X_i, \\ \bigcup_{i=1}^{r}X_i=X, \, X_i\cap X_j=\emptyset \text{ for } i\neq j} \right\}.
		 	\]
		 	We define the \textit{$X$-th tensor power of $k^r$} as the Schwartz space
		 	\[
		 	(k^r)^{\otimes X}\coloneqq \cC(\pa_r(X))\coloneqq \{v\colon \pa_r(X)\to k \mid \exists U\overset{\text{open}}{\subseteq} G \colon U\cdot v=v\}.
		 	\]
		 \end{defA}
		 %Roughly speaking, the set $\pa_r(X)$ encodes ways to write $X$ as a union of $r$-many good subsets, and finite subsets are good. 
		 The Schwartz space in the context of oligomorphic groups was introduced in \cite{harmansnowden22}, and serves as a particularly well-behaved linearization. Most importantly its subspace of $G$-invariants has a natural basis labeled by $G$-orbits on $\pa_r(X)$.
		 If $X$ is finite, $(k^r)^{\otimes X}$ is the usual tensor power $(k^r)^{\otimes |X|}$, cf.\ Example \ref{example: finite Y}.
		 \begin{thmA}[Theorem \ref{theorem lie algebra action} and \ref{theorem: Stanley's argument}] \label{theorem B}
		 	Let $\mu$ be a measure of $G$ with values in $k$ in the sense of \cite{harmansnowden22}. 
		 	\begin{enumerate} 
		 		\item There is a natural action of the Lie algebra $\gl_r(k)\acts (k^r)^{\otimes X}$ depending on $\mu$, which commutes with the action of $G$.
		 		\item In the special case $r=2$, $k=\C$, this gives an action of $\sl_2(\C)\acts \H_{G,X}^{\star}$ on the orbit-algebra.
		 		If $X$ is infinite, then this representation has a (possibly infinite) ascending filtration whose subquotients are lowest weight $\sl_2(\C)$-Verma modules.
		 	\end{enumerate}
		 \end{thmA}
		 The filtration by Verma modules gives representation theoretic meaning to Cameron's result about monotonicity of the infinite sequence $\left(\left| G\backslash \binom{X}{n} \right|\right)_{n\in \mN}$, since each Verma module corresponds to the constant $1$-sequence $1,1,1,1,\ldots$. This unifies Stanley's and Cameron's approaches. 
		 
		 \begin{beisA}
		 	Consider the the infinite symmetric group $G=S_{\infty}$ acting on $X=\N$. Measures $\mu$ on $S_{\infty}$ correspond to scalars $\lambda=\mu(\N)\in \C$. The $\sl_2(\C)$-action on $\H^{\star}_{G,X}$ looks as follows:
		 	\[
		 	\sl_2(\C) \acts \H_{S_{\infty},\N}^{\star}\colon 
		 	\begin{tikzcd}
		 		{v_{[0]}} \arrow[r, red!60!black, "1", bend left] \arrow["-\lambda"', teal!60!black, loop, distance=2em, in=305, out=235] & {v_{[1]}} \arrow[l, blue!60!black,  "\lambda", bend left] \arrow[r, red!60!black, "2", bend left] \arrow["-\lambda+2"', teal!60!black, loop, distance=2em, in=305, out=235] & {v_{[2]}} \arrow[l, blue!60!black, "\lambda-1", bend left] \arrow[r, red!60!black, "3", bend left] \arrow["-\lambda+4"', teal!60!black, loop, distance=2em, in=305, out=235] & {v_{[3]}}  \arrow[l, blue!60!black,  "\lambda-2", bend left] \arrow[r, red!60!black, "4", bend left] \arrow["-\lambda+6"', teal!60!black, loop, distance=2em, in=305, out=235] & \cdots \arrow[l, blue!60!black, "\lambda-3", bend left]
		 	\end{tikzcd}
		 	\]
		 	Here $v_{[n]}=\sum_{S\in\binom{\N}{n}}v_S$ is the basis vector corresponding to the unique $S_{\infty}$-orbit on $\binom{\N}{n}$. In this case $\H_{S_{\infty},\N}^{\star}$ is the lowest weight Verma module $M^{-}(-\lambda)$. The action of $\textcolor{red!60!black}{e}=\begin{psmallmatrix}
		 		0 & 1 \\ 0 & 0 
		 	\end{psmallmatrix}\in \sl_2(\C) \acts \H^{\star}_{G,X}$ is multiplication with $v_{[1]}\in H^{1}_{G,X}$. The actions of $\textcolor{blue!60!black}{f}=\begin{psmallmatrix}
		 		0 & 0 \\ 1 & 0
		 	\end{psmallmatrix}$ and $\textcolor{teal!60!black}{h}=\begin{psmallmatrix}
		 		1 & 0 \\ 0 & -1
		 	\end{psmallmatrix}$ are new and depend on the choice of measure $\mu$.
		 \end{beisA} 
	 \begin{rema*}
	 	One of the motivations of the introduction of measures in \cite{harmansnowden22} was to generalize the construction of Deligne's interpolation category $\operatorname{\underline{Rep}}(S_t)$ from \cite{deligne07}. Our construction $(k^r)^{\otimes X}$ is a generalization of the infinite tensor power in work by Entova-Aizenbud in \cite{aizenbud2015} on Schur--Weyl duality for interpolation categories.	 
	 \end{rema*}
%		 
%		 All the actions are appropriately compatible with the permutation action of $S_r$ on $\pa_r(X)$. Also the basis $\bigoplus_{n\geq 0} \cC\left( \binom{X}{n} \right)^{G}\subseteq \cC(\pa_2(X))$ is labeled by orbits $\bigsqcup_{n\geq 0}G\backslash \binom{X}{n}$.
%		  
The final theorem explains how to combine measures with the regular measure on $\mQ$.
\begin{thmB}[Theorem \ref{theorem measure on Fibonacci}]
	Let $\cS$ be a homogeneous oligomorphic graph. Finite subgraphs of $\cS\times \mQ$ correspond to sequences $(\cs_1,\ldots, \cs_l)$ of finite subgraphs $\cs_1,\ldots, \cs_l\subseteq \cS$. Moreover if $\cS$ admits an R-measure $\nu$, then $\cS\times \mQ$ admits an R-measure $\nu_{\mQ}$ which assigns to $(\cs_1,\ldots, \cs_l)$ the value $(-1)^l\prod_{i=1}^{l}\nu(\cs_i)$.
\end{thmB}
In the case $\cS=K_m$ we obtain a measure on $G=\Aut(K_m\times \mQ)$. Here $\dim_k \H_{G,X}^{n}=a_n$ gives the $m$-bonacci sequence. We showcase the $\sl_2(\C)$-action on $\H_{G,X}^{\star}$ in the Fibonacci case in Example \ref{example Verma of fibonacci}.

%		\subsection*{A Personal Recommendation and Short Overview}
%		We cannot recommend strongly enough that the reader takes a look at \cite[Lecture 1]{reinernote}, which discusses the connection between finite sequences of orbits and tensor powers. One does not need to be familiar with oligomorphic groups, homogeneous structures, or measures to read our paper. \Cref{section preliminaries} offers an overview of these topics. We also don't expect the reader to be familiar with Schwartz spaces. These are explained in \Cref{section linearizing}. We will sometimes use different notation from \cite{harmansnowden22} to avoid integrals and to make the text easier to read. All of the new results are contained in \Cref{section Lie algebras}.
		\subsection{Acknowledgments}
		 Thanks go to Sebastian Meyer for many discussions on oligomorphic groups and measures, Manuel Bodirsky for introducing me to oligomorphic groups and being my postdoc mentor, Pierre Touchard and Matej Konečný for model theory explanations, Johannes Flake for a computation in his office, Jonathan Wiebusch for combinatorics, Mateusz Stroiński for proofreading, Catharina Stroppel for an important clarification on Vermas, Peter Cameron for a fun chat, and Ulrich Krähmer for letting me keep my desk. Final thanks go to the Max--Planck institute in Bonn for coffee breaks, Liao Wang for listening to me, and my friends Thomas Häßel, Till Wehrhan, Benjamin and Janna Nettesheim, Timm Peerenboom, who let me sleep on their couches in Bonn and Aachen.
		 \subsection{Tool and computational resource disclosure}
		 LLM's or other AI tools were not used to generate the mathematical content of this paper. We used ChatGPT and Gemini for more complicated \LaTeX{} commands, specifically Tikz (in Figure \ref{the action of e} and within Example \ref{example Verma of fibonacci}), and to check grammar.
%		 \subsection*{Use of AI}
%		 We used ChatGPT/DeepSeek to create some of the Tikz pictures, help with \LaTeX, grammar, and to search for related literature.
		 \newpage
		 \tableofcontents
	\section{Combinatorial Preliminaries} \label{section preliminaries}
	We denote by $\mN=\{0,1,2,\ldots\}$ the set of non-negative integers and by $\N=\{1,2,\ldots\}$ the set of all positive integers. We fix a set $X$. For a set $X$ and $n\in \mN$ we write $\binom{X}{n}\coloneqq \{\{x_1,\ldots, x_n\}\subseteq X\}$ for the set of $n$-element subsets of $X$.
	\subsection{Oligomorphic Permutation Groups}
	 \label{subsection on Oligomorphic Permutation Groups}
	 This section is based on \cite{cameron90}.
	\begin{defi}
		Let $G$ be a group. Let $Y$ be a $G$-set. We denote by $G\backslash Y$ the set of $G$-orbits. We call $Y$ \textit{orbit-finite}, if $|G\backslash Y|<\infty$.
	\end{defi}
	We fix a set $X$.
	\begin{defi}
		We denote by $\Sym(X)$ the group of all self-bijections of $X$.
		In the case $X=\N$ we write $S_{\infty}\coloneqq\Sym(\N)$ and we call it the \textit{infinite symmetric group}. 
	\end{defi}
	The following definition is due to Peter Cameron.
	\begin{defi}
		A subgroup $G\subseteq \Sym(X)$ is called a \textit{permutation group}. A permutation group $G\subseteq \Sym(X)$ is called \textit{oligomorphic}, if the $G$-set $X^n$ with componentwise action is orbit-finite for all $n\in \mN$. 
	\end{defi}
	\begin{rema}
		``Oligo'' means ``few'' and ``morph'' means ``shape''. If $k$ is a field, one can identify the number $|G\backslash X^n|$ of orbits 
		with the dimension of the space of homomorphisms $\Hom_G(k_{\operatorname{triv}}, \Maps(X^n, k))$, where $k_{\operatorname{triv}}$ denotes $k$ with the trivial $G$-action. Hence $G$ is oligomorphic if and only if there are `few morphisms'. In Section \ref{section linearizing} we will discuss linearizations in detail.
	\end{rema}
\begin{beis} \label{example Bell number Stirling number}
	Every permutation group of a finite set is oligomorphic. The trivial group $G=\{\id_X\}\subseteq \Sym(X)$ is oligomorphic if and only if $X$ is finite. The infinite symmetric group $S_{\infty}=\Sym(\N)$ is the textbook example of an infinite oligomorphic permutation group. 
	For $n\in \mN$ the set $\N^n$ has orbits labeled by set-partitions of $\{1,\ldots, n\}$. Their number is the \textit{Bell number} $B_n$, which decomposes as sum $B_n=\sum_{i=1}^{n}\operatorname{St}_{n,i}$ of \textit{Stirling numbers of the second kind}. These count the number of ways to partition an $n$ element set into $i$-many parts. 
\end{beis}

\begin{lemm} \label{lemma oligomorphic}
	Let $G\subseteq \Sym(X)$ be a permutation group. Then $G$ is oligomorphic if and only if the $G$-set $\binom{X}{n}$ is orbit-finite for all $n\in \mN$. 
\end{lemm}
\begin{proof}
	Clearly if $X^n$ is orbit-finite, then
	\[
		X^{(n)}\coloneqq \{(x_1,\ldots, x_n) \mid x_i\neq x_j \text{ for } i\neq j\}\subseteq X^n
	\] 
	is orbit-finite. Hence $\binom{X}{n}\cong S_n\backslash X^{(n)}$ also is orbit-finite, since it is a quotient of $X^{(n)}$. For the other direction we recall an argument from \cite[\S1.2]{cameron90}. We have
	\[
	|G \backslash X^{(n)}|\leq n! \cdot \left|G \backslash \binom{X}{n}\right|< \infty.
	\]
	Furthermore as $G$-sets $X^n\cong \bigsqcup_{i=1}^n \operatorname{St}_{n,i} X^{(i)}$, where $\operatorname{St}_{n,i}$ is the Stirling number from Example \ref{example Bell number Stirling number}.  
	We obtain
	\[
	|G \backslash X^{n}|= \sum_{i=1}^{n} \operatorname{St}_{n,i} |G \backslash X^{(i)}| < \infty. \qedhere
	\]
\end{proof}
\begin{defi} \label{definition rank generating function for group}
	Let $G\subseteq \Sym(X)$ be an oligomorphic permutation group. We define the \textit{rank generating function} of $G\subseteq \Sym(X)$ as the formal power series 
	\[
		r_{G}(q)\coloneqq \sum_{n=0}^{\infty} \left|G\backslash \binom{X}{n}\right|q^n\in \mN[[q]].
	\]
\end{defi}
\begin{lemm} \label{lemma disjoint unions is oligomorphic}
	Let $X_1$, $X_2$ be sets and $G_1\subseteq \Sym(X_1)$ and $G_2\subseteq \Sym(X_2)$ be oligomorphic permutation groups. The product $G_1\times G_2\subseteq \Sym(X_1\sqcup X_2)$ is also oligomorphic and $r_{G_1\times G_2}(q)=r_{G_1}(q)\cdot r_{G_2}(q)$.
\end{lemm}
\begin{proof}
	Let $n\in \mN$. We have $\binom{X_1\sqcup X_2}{n}\cong \bigsqcup_{i=0}^{n}\binom{X_1}{i}\times \binom{X_2}{n-i}$ as $G_1\times G_2$-sets, which gives the statement about the product of rank generating functions. In particular $G_1\times G_2$ is oligomorphic by Lemma \ref{lemma oligomorphic}.
\end{proof}
\begin{beis} \label{example: polynomial ring generating function}
	We have $r_{S_{\infty}}(q)=\frac{1}{1-q}=1+q+q^{2}+\cdots$, since $S_{\infty}$ acts transitively on $\binom{\N}{n}$ for all $n\in \mN$. For $r\in \N$ we can consider the oligomorphic group $S_{\infty}^m\subseteq \Sym(\N\times \{1,\ldots, m\})$. Here we get the series
	\[
		r_{S_{\infty}^m}(q)=\left(\frac{1}{1-q}\right)^m=\sum_{n=0}^{\infty} \binom{n+m-1}{n}q^n.
	\]
	It is the Hilbert--Poincaré series of the polynomial ring in $m$-variables $k[x_1,\ldots, x_m]\cong k[x]^{\otimes m}$. We will return to this example later in Subsection \ref{subsection disjoint unions}.
\end{beis}

\begin{beis} \label{beis: partitions}
	Let $X=\mQ\times \N$. Consider the group given by the wreath product $\Aut(\mQ, <) \wr S_{\infty} \subseteq \Sym(\mQ\times \N)$. It may be imagined as follows: The set $X=\mQ\times \N$ consists of countably many copies $\mQ\times \{n\}$ for $n\in \N$. These copies are permuted by $G$ externally in an arbitrary way (this is the action of $S_{\infty}$), and internally via order-preserving automorphisms of $\mQ$. There is a bijection
	\[
		G\backslash\binom{X}{n} \xleftrightarrow{1\colon 1} \{\lambda \mid \lambda \vdash n\} 
	\]
	between orbits and integer partitions $\lambda=\lambda_1\geq \lambda_2\geq\cdots \geq \lambda_r>0$ of $n$.
	Indeed, each $G$-orbit on $\binom{X}{n}$ may be represented by the $n$-element set
	\[
		\{(1,1),\ldots, (\lambda_1,1), \, (1,2),\ldots, (\lambda_2,2), \,  \ldots, \,  (1,r),\ldots, (\lambda_r,r)\}
	\]
	for some $\lambda=\lambda_1\geq \lambda_2\geq \lambda_r>0$ with $n=\sum_{i=1}^r\lambda_i$. This uses that the action of $\Aut(\mQ,<)$ on $\binom{\mQ}{m}$ is transitive for all $m\in \mN$, since any $m$-element set can be imagined as totally ordered in the first place.
	The rank generating function of $G$ is the usual generating function of integer partitions
	\[
	r_G(q)=\prod_{i=1}^{\infty}\frac{1}{1-q^{i}}=1+q+2q^{2}+3q^{3}+5q^{4}+7q^{5}+11q^{6}+\cdots.
	\]
\end{beis}
\begin{rema}
	There are many other infinite integer sequences, which appear as orbit numbers for oligomorphic groups. Highlights include Fibonacci numbers, Catalan numbers, and the number of partitions with at most $n$ rows (for fixed $n\in \N$). See \cite{cameron00} for an overview of sequences of group orbits which have an OEIS entry.
\end{rema}

\subsection{Topology on Oligomorphic Permutation Groups} \label{subsection topology on oligomorphic permutation groups}

We fix a set $X$ and an oligomorphic permutation group $G\subseteq \Sym(X)$. This section is based on \cite{cameron90,tsankov12,harmansnowden22}.
	\begin{defi} \label{defi topology}
		Subsets of $G$ of the form
		\[
		U(x_1,\ldots, x_n; x_1',\ldots, x_n')\coloneqq \{g\in G \mid gx_i=x_i' \text{ for all } i\in \{1,\ldots, n\}\},
		\]
		where $n\in \mN$, $x_1,\ldots, x_n, x_1',\ldots, x_n'\in X$, are called \textit{basic open sets} of $G$. Every subset of $G$, which is a union of basic open sets, is called \textit{open}. This defines the \textit{pointwise (convergence) topology} on $G$. We set 
		\[
			U(x_1,\ldots, x_n)\coloneqq U(x_1,\ldots, x_n; x_1,\ldots, x_n).
		\]
	\end{defi}
	The next remark is the conceptual explanation for the existence of this topology.
	\begin{rema}
		Start with a set $X$ and view it as discrete topological space. The set $\Maps(X,X)=\{f\colon X\to X\}$ may be identified with the product $\prod_{x\in X}X$, i.e.\ carries a natural topology, which is the product topology. 
		In particular $\Sym(X)$ inherits the subspace topology of $\Maps(X,X)$. The subspace topology of a permutation group $G\subseteq \Sym(X)$ is precisely the pointwise topology. Note that the (infinite) product of discrete topological spaces is in general no longer discrete. This is the same reason, why the topology on $\mathbb{Z}_p$ or the Cantor set are non-trivial.
	\end{rema}
	We gather the most important properties known about the pointwise topology.
	\begin{prop} \label{topological properties}
		The pointwise topology turns $G$ into a topological group, which is
		\begin{enumerate}
			\item Hausdorff,
			\item \textit{non-Archimedean}, i.e.\ every open neighborhood of $1\in G$ contains an open subgroup
			\item \textit{Roelcke precompact}, i.e.\ for each two open subgroups $U_1, U_2$ the number $|U_1\backslash G / U_2|$ of double cosets is finite.
		\end{enumerate}
		Additionally the following properties hold about subgroups.
		\begin{enumerate}
			\item[a)] A subgroup $U\subseteq G$ is open, if and only if it contains a basic open subgroup $U(x_1,\ldots, x_n)$ for some $n\in \mN$ and $x_1,\ldots, x_n\in X$. \label{item condition a}
			\item[b)] Each open subgroup $U\subseteq G$ is an oligomorphic permutation group of $X$, moreover the pointwise topology on $U$ agrees with the subspace topology of the pointwise topology on $G$. \label{item condition b}
		\end{enumerate}
	\end{prop}
	\begin{proof}
		For Roelcke precompactness note that the $G$-set $G/U_2$ is orbit-finite (it consists of one orbit), and hence it also orbit-finite as a $U_1$-set, since $U_1$ contains a basic open. The other properties are straightforward.
	\end{proof}
	
	From now on we fix the pointwise topology on $G$.
	\begin{defi}
		Let $Y$ be a $G$-set. We call $Y$ \textit{smooth}, if the action map $G\times Y\to Y$ is continuous, when $Y$ is equipped with the discrete topology. In concrete terms this means that the \textit{point stabilizer subgroup} $G_y\coloneqq \{g\in G \mid gy=y\}\subseteq G$ of each $y\in Y$ is open. 
	\end{defi}
	\begin{beis} \label{example smooth orbit-finite sets}
		By the orbit-stabilizer theorem smooth and orbit-finite $G$-sets are those which are isomorphic to
		$G/U_1\sqcup \cdots \sqcup G/U_r$
		for some $r\in \mN$ and open subgroups $U_1,\ldots, U_r\subseteq G$.
		Common examples of smooth, orbit-finite $G$-sets are the one-point set $\{\pt\}=G/G$, $X$ and more generally $X^n$ for all $n\in \mN$. The stabilizers of points in these examples are basic open subgroups.
		
		Every $G$-subset of a smooth $G$-set is smooth, for instance
		\begin{equation*}
			X^{(n)}\coloneqq \{(x_1,\ldots, x_n)\in X^n \mid x_i \neq x_j \text{ for all } i\neq j\}\subseteq X^n
		\end{equation*}
		is a smooth $G$-set. Every quotient of a smooth $G$-set is smooth by property a) in Proposition \ref{topological properties}. For instance, consider for each $n\in \mN$ the set $\binom{X}{n}$
		with elementwise $G$-action. This $G$-set is smooth, since there is a surjective homomorphism of $G$-sets from $X^{(n)}$ onto it. The stabilizer of a point $\{x_1,\ldots, x_n\}\in \binom{X}{n}$ is precisely $(\Sym(\{x_1,\ldots, x_n\})\times U(x_1,\ldots, x_n))\cap G$ and contains $U(x_1,\ldots, x_n)$. One can show that every orbit on a smooth $G$-set is a quotient of an orbit which appears in $X^n$ for some $n\in \mN$. 
	\end{beis} 
	\begin{lemm}[{\cite[Proposition 2.8]{harmansnowden22}}] \label{lemma products of smooth}
		Let $Y_1$ and $Y_2$ be smooth $G$-sets. The disjoint union $Y_1\sqcup Y_2$ and Cartesian product $Y_1\times Y_2$ are also smooth. If $Y_1$ and $Y_2$ are smooth and orbit-finite, then $Y_1\sqcup Y_2$ and $Y_1\times Y_2$ are orbit-finite.
	\end{lemm}
\begin{proof}
	The statements are obvious for disjoint unions. The stabilizer subgroup of $(y_1,y_2)\in Y_1\times Y_2$, is the intersection $G_{y_1}\cap G_{y_2}$, which is open. To conclude that $Y_1\times Y_2$ is orbit-finite, assume without loss of generality that $Y_1$ and $Y_2$ are transitive. Applying the same argument as in Proposition \ref{topological properties} for Roelcke precompactness gives the proof. 
\end{proof}
	\subsection{Measures on Groups}
	We fix a set $X$ and an oligomorphic permutation group $G\subseteq \Sym(X)$. This section is based on \cite{harmansnowden22}. 
	\begin{defi} \label{definition G hat sets}
		We define:
		\begin{enumerate}
			\item A \textit{$\hat{G}$-set} is a set $Z$ together with a smooth and orbit-finite action of some open subgroup $U\subseteq G$. Shrinking the subgroup does not change a $\hat{G}$-set.
			\item Let $Y$ be a smooth $G$-set. A subset $Z\subseteq Y$ is called a \textit{$\hat{G}$-subset}, if it is stable under the action of some open subgroup $U\subseteq G$.
			We denote by $2_{\hat{G}}^{Y}\subseteq 2^{Y}$ the set of $\hat{G}$-subsets of $Y$.
		\end{enumerate}
	\end{defi}
	\begin{lemm} \label{lemma elementary set-theoretic properties of hat G sets}
		 Let $Y$ be a smooth $G$-set. The set of $\hat{G}$-subsets $2_{\hat{G}}^{Y}$ is closed under taking complements, finite unions and finite intersections. In other words it forms a boolean algebra. Moreover if $y\in Y$, then $Z\cup\{y\}$ and $Z\setminus\{y\}$ are again $\hat{G}$-subsets of $Y$. 
	\end{lemm}
	\begin{proof}
		Under these operations stabilizers of points are not changed, also the intersection of open subgroups is open.
	\end{proof}
\begin{beis} \label{example finite or cofinite subsets}
	Let $G=S_{\infty}$. The $\hat{G}$-subsets of $\N$ are those subsets which are finite or cofinite, i.e.\ whose complement is finite.
\end{beis}
	\begin{defi}[{\cite[Definition 3.1]{harmansnowden22}}]
		A \textit{measure} on $G$ with values in the field $k$ is a rule $\mu$ assigning each $\hat{G}$-set $Z$ a value $\mu(Z)\in k$ such that the following assertions hold:
		\begin{enumerate}
			\item Isomorphism invariance: $\mu(Z_1)=\mu(Z_2)$ holds for all $\hat{G}$-sets $Z_1,Z_2$ with $Z_1\cong Z_2$.
			\item Normalization: $\mu(\{\pt\})=1$.
			\item Conjugation invariance: $\mu(Z^g)=\mu(Z)$ holds for each $\hat{G}$-set $Z$ and all $g\in G$. Here $Z^g$ denotes the $\hat{G}$-set obtained by conjugating the action on $Z$ with $g\in G$.
			\item Multiplicativity: $\mu(Z_1)=\mu(f^{-1}(z)) \mu(Z_2)$ holds for all homomorphisms $Z_1\xrightarrow{f} Z_2$ of transitive $U$-sets for some open subgroup $U\subseteq G$ and points $z\in Z_2$. 
		\end{enumerate}
	A measure is called \textit{regular} if $\mu(Z)\neq 0$ for all transitive $G$-sets $Z$.
	\end{defi}
	\begin{beis} \label{example measure in finite setting}
		If $X=\{x_1,\ldots, x_n\}$ is finite, then $G\subseteq \Sym(X)$ is finite and discrete.
		Smooth and orbit-finite $\hat{G}$-sets are exactly finite sets, since the trivial group $\{1\}\subseteq G$ is open (it is precisely $U(x_1,\ldots, x_n)$) and acts on each finite set. There is a unique measure on $G$ with values in $k$, which assigns each finite set $Z$ the value $|Z|\in k$ viewed as a scalar. This measure on $G$ is regular if and only if $\operatorname{char}(k)$ does not divide the order of $G$, since $G=G/\{1\}$ is a transitive smooth $G$-set in this setting. 
	\end{beis}
	\begin{beis}[{\cite[\S14.6]{harmansnowden22}}] \label{example measure for S_infty}
		Assume $\operatorname{char}(k)=0$. For each $\lambda\in k $ there is a measure $\mu_{\lambda}$ on $S_{\infty}$, which is completely determined by assigning to $\N$ the value $\mu(\N)=\lambda$. This measure is regular if and only if $\lambda\notin \mN$. We will explain this example in more detail in Example \ref{example measure for S} and keep coming back to it over and over. For the reader familiar with Deligne's interpolation category (cf.\ \cite{deligne07}): it will become clear why we use the letter $\lambda\in k$ instead of $t\in k$ in Subsection \ref{subsection: sl2 and Stanley}.
	\end{beis}
	\begin{beis}[{\cite[\S16.7]{harmansnowden22}}] \label{example measure on Q}
		Let $k$ be any field. There is a unique regular measure on $\Aut(\mQ, <)$. It assigns to $\mQ$ the value $\mu(\mQ)=-1$. There are three more measures, which are not regular. The following is some heuristics for the regular measure on $\mQ$. If one removes a point from $\mQ$, one gets two copies of $\mQ$. The intuition should be $\mu(\mQ)-1=2\cdot \mu(\mQ)$. The unique solution of this equation is $\mu(\mQ)=-1$.
	\end{beis}
	\begin{rema}
		 Let $G\subseteq \Sym(X)$ be an oligomorphic permutation group and assume $X$ is countably infinite. Every $\widehat{\Sym(X)}$-set is in particular a $\hat{G}$-set. In particular every measure on $G$ really is a finer version of one of the measures $\mu_{\lambda}$ for $\Sym(X)\cong S_{\infty}$ in Example \ref{example measure for S_infty}.
	\end{rema}
	\begin{beis}
		Since the publication of \cite{harmansnowden22} many papers were written concerned with computing measures for various oligomorphic groups, often connected to some family of combinatorial objects. See \cite{harmansnowdensnyder23} for the group of permutations of roots of unity preserving the cyclic order, see \cite{harmansnowdeninterpolation} for $\GL_{\infty}(\mathbb{F}_q)$ and other classical groups, \cite{canrud26} for planar trees, \cite{, krizsophietree} for trees with ordered vertices, \cite{snowden24} for finite sets with two total orders, and \cite{krizsophiequantum} for a quantum $\Aut(\mQ,<)$. 
	\end{beis}
	A measure should be thought of as a way to ``count'' infinite sets $Z$ by assigning them a value $\mu(Z)$. In the next section we discuss automorphism groups of homogeneous graphs $\cS$. In that context the usual $G$-set $Y$ is the set of all embeddings $\{\cs\hookrightarrow \cS\}$ of a finite graph $\cs$.
	\begin{rema}
		In \cite[\S3.8]{harmansnowden22} it is shown that giving a measure is equivalent to giving to a so-called generalized index $[[U \colon V]]$ on $G$ for open subgroups $V\subseteq U\subseteq G$. This notion generalizes the usual index for subgroups and can be helpful for intuition.
	\end{rema}
	
	\subsection{Graphs with Several Families of Edges} \label{subsection graphs}
	Graphs are a template for combinatorial objects, whose isomorphism classes may be counted in terms of orbits. Since some of our examples cannot be expressed as ordinary directed graphs, we will work with multi-relational graphs (also known as edge-colored graphs), which are directed graphs with several families of edges.
	\begin{defi}
		Let $I$ be a set. A \textit{($I$-multi-relational) graph} $\cS$ consists of a (possibly infinite) set $X$ of \textit{vertices} together with (possibly infinite) subsets $E_i\subseteq X^2$ of \textit{(oriented) $i$-edges} for each $i\in I$.
	\end{defi}
	\begin{beis}
		An $I$-multi-relational graph for $I=\emptyset$ is simply a set. The most important $\emptyset$-multi-relational graph is $\cS=\N$.
		An $I$-multi-relational graph for $I=\{1\}$ is a directed graph $(X,E)$ without parallel edges, but possibly loops $(x,x)$ at vertices $x\in X$. An example for a $\{1\}$-multi-relational graph is $(\mQ,<)$. Its vertex set $X=\mQ$ consists of rational numbers and its set of edges $E=\{(x_1,x_2) \mid x_1< x_2\}\subseteq \mQ^{2}$ describes the total order.
	\end{beis}
	\begin{nota}
		We will fix $I$, but keep it implicit from now on, and refer to $I$-multi-relational graphs simply as \textit{graphs}.
	\end{nota}
\begin{rema}
	We could use arbitrary relational structures in this paper instead of (multi-relational) graphs. If the reader is familiar with model theory, they should replace every graph with ``relational structure''.
\end{rema}
	\begin{defi}
		Let $\cS=(X,(E_i)_{i \in I})$ be a graph.
		%and we write $|\cS|\coloneqq |X|$ for their number. 
		Each subset $X'\subseteq X$ defines an \textit{induced subgraph} $\cS'$ whose vertices are $X'$ and whose $i$-edges are $E_i\cap (X')^2$. 
		For us \textit{subgraph} will always refer to an induced subgraph. 
		The graph with no vertices is denoted $\emptyset\in \age(\cS)$.
	\end{defi}
	\begin{defi}
		Let $\cS$, $\cS'$ be ($I$-multi-relational) graphs. A homomorphism $f\colon \cS\to \cS'$ is a map $X\to X'$ mapping $i$-edges to $i$-edges for all $i\in I$. An \textit{embedding of graphs} $f\colon \cS\hookrightarrow \cS'$ is an injective homomorphism, which also preserves non-edges. An \textit{isomorphism} $f\colon \cS\xrightarrow{\cong} \cS'$ is a bijective embedding. An \textit{automorphism} is a self-isomorphism. We denote by $\Aut(\cS)$ the \textit{automorphism group} of $\cS$. We view it as subgroup of $\Sym(X)$, the self-bijections of the vertices $X$.
	\end{defi}
	\begin{defi}
		The \textit{age of a graph} $\cS$ is the class
		\[
		\age(\cS)\coloneqq\{\cs \mid \text{$\cs$ finite graph}, \, \text{s.t.\ there exists embedding } \cs\hookrightarrow \cS \},
		\]
		of finite graphs, which embed into $\cS$.
	\end{defi}
	\begin{defi}
		A graph $\cS$ is called \textit{homogeneous} if for every two embeddings $\cs_1 \hookrightarrow \cS$, $\cs_2 \hookrightarrow \cS$ of finite graphs $\cs_1, \cs_2$ every isomorphism $f\colon \cs_1\to \cs_2$ can be lifted to an automorphism $\tilde{f}\in \Aut(\cS)$. Equivalently $\cS$ is called \textit{homogeneous} if for every $\cs\in \age(\cS)$ the natural action of $\Aut(\cS)$ on the \textit{set of embeddings} $\cS^{[\cs]}\coloneqq\{\iota\colon \cs\hookrightarrow \cS\}$ via post-composition is transitive.
	\end{defi}
	The following standard lemma gives the importance of homogeneity from an enumerative combinatorics point of view.
	\begin{lemm}[\cite{cameron90}] \label{lemma bijection isoclasses and orbits}
		Let $\cS$ be a homogeneous graph. Let $n\in \mN$. Then there is a one-to-one correspondence
		\[
			\{\cs\in \age(\cS)\mid |\cs|=n\}/{\cong} \, \xleftrightarrow[]{1\colon 1}\, 	\Aut(\cS) \backslash \binom{X}{n}
		\]
		between isomorphism classes of $n$-element graphs, which admit an embedding into $\cS$, and $\Aut(\cS)$-orbits on the set of $n$-element subsets of the vertices $X$ of $\cS$.
	\end{lemm}
	\begin{koro} \label{corollary oligomorphic}
		Let $\cS$ be a homogeneous graph. Assume that for all $n\in \mN$ only finitely many isomorphism classes of subgraphs $\cs\hookrightarrow \cS$ with $n$ vertices exist. Then $\Aut(\cS)\subseteq \Sym(X)$ is an oligomorphic permutation group.
	\end{koro}
\begin{proof}
	By Lemma \ref{lemma bijection isoclasses and orbits} we have $|\Aut(\cS) \backslash \binom{X}{n}|<\infty$ for all $n\in \mN$. Now apply Lemma \ref{lemma oligomorphic}.
\end{proof}
\begin{defi} \label{definition homogeneuous oligomorphic}
	A graph $\cS$ is called \textit{homogeneous oligomorphic}, if it satisfies the following conditions:
	\begin{enumerate}
		\item \label{affe1} $\cS$ is homogeneous, \label{homogeneuous}
		\item \label{affe2} it has only finitely many isomorphism classes of $n$-element subgraphs for each $n\in \mN$. \label{oligomorphic}
	\end{enumerate}
	We define the \textit{rank-generating function of $\cS$} to be $r_{\cS}(q)\coloneqq r_{\Aut(\cS)}(q)$, the rank-generating function of its automorphism group cf.\ Definition \ref{definition rank generating function for group}.
\end{defi} 
From now on we fix a homogeneous oligomorphic graph $\cS$. In particular $\Aut(\cS)\subseteq \Sym(X)$ is an oligomorphic group.
	\begin{defi}
			Let $\cs_0,\cs_1,\cs_2\in \age(\cS)$ be finite graphs. Moreover let $\iota_1\colon \cs_0\hookrightarrow \cs_1$ and $\iota_2\colon\cs_0\hookrightarrow \cs_2$ be embeddings of graphs. An \textit{amalgamation} of $\iota_1$ and $\iota_2$ in $\age(\cS)$ consists of a finite subgraph $\cs_1\cup_{a, \cs_0}\cs_2\hookrightarrow \cS$ which admits embeddings $j_1\colon \cs_1 \hookrightarrow \am$, $j_2\colon \cs_2\hookrightarrow \am$ such that the diagram
			\[
			\begin{tikzcd}
				\cs_0 \arrow[r, "\iota_1", hook] \arrow[d, "\iota_2"', hook] & \cs_1 \arrow[d, "j_1", hook] \\
				\cs_2 \arrow[r, "j_2", hook]                               & \am.                       
			\end{tikzcd}
			\]
			commutes and such that every vertex of $\cs_1\cup_{a,\cs_0} \cs_2$ is in the image of $j_1$ or $j_2$, i.e.\ $j_1$ and $j_2$ are jointly surjective. An amalgamation is called a \textit{one-point amalgamation} if
			\[
				|\cs_1\setminus \iota_1(\cs_0)|=1=|\cs_2\setminus \iota_2(\cs_0)|.
			\]
	\end{defi} 
	\begin{defi}
		 	Let $\cs_0,\cs_1,\cs_2\in \age(\cS)$. Let $\iota_1\colon \cs_0\hookrightarrow \cs_1$, $\iota_2\colon \cs_0\hookrightarrow \cs_2$ be embeddings. Let $\mathsf{a}=(\am, j_1, j_2)$, $\mathsf{a}'=(\cs_1\cup_{a', \cs_0}\cs_2, j_1', j_2')$ be two amalgamations of $(\iota_1,\iota_2)$ in $\cS$. An \textit{isomorphism of amalgamations} $\ca\to \ca'$ is an isomorphism $\varphi\colon \cs_1\cup_{a, \cs_0}\cs_2\xrightarrow{\cong} \cs_1\cup_{a', \cs_0}\cs_2$ of graphs such that the following diagram commutes
		 	\[
		 		\begin{tikzcd}
		 			& \cs_1 \arrow[d, "j_1"', hook] \arrow[rdd, "j_1'", hook, bend left] &                      \\
		 			\cs_2 \arrow[r, "j_2", hook] \arrow[rrd, "j_2'"', hook, bend right] & { \cs_1\cup_{a, \cs_0}\cs_2} \arrow[rd, "{\varphi, \cong}"']             &                      \\
		 			&                                                                  & { \cs_1\cup_{a', \cs_0}\cs_2}.
		 		\end{tikzcd}
		 	\]
	\end{defi}
	Note that two amalgamations, whose underlying graphs are isomorphic may not be isomorphic as amalgamations.
	
	\begin{rema}
		Every homogeneous oligomorphic graph $\cS$ automatically satisfies in addition to Definition \ref{definition homogeneuous oligomorphic}, \ref{affe1}, \ref{affe2} the following finiteness property for the number of amalgamations:
		\begin{enumerate}
			\item[(3)] For every $\cs_0, \cs_1, \cs_2\in \age(\cS)$ and embeddings $\iota_1\colon \cs_0\hookrightarrow \cs_1$, $\iota_2\colon \cs_0\hookrightarrow \cs_2$ there exists an amalgamation $(\am, j_1,j_2)\in \age(\cS)$. Moreover only finitely many isomorphism classes of such amalgamations exist.
		\end{enumerate}
	\end{rema}

	\subsection{Measures on Graphs} \label{subsection: measures on graphs}
	Here we recall a different perspective on measures for automorphism groups of homogeneous graphs. Everything in this section except the one-point lemma (\ref{technical lemma}) is contained in \cite[\S6]{harmansnowden22} with more model-theoretic language.
	We fix a homogeneous oligomorphic graph $\cS$, whose vertices are denoted $X$. The following notion is the analogue of regular measure for graphs.
	\begin{defi}[\cite{harmansnowden22}] \label{defi R-measure}
		An R-measure for $\cS$ with values in $k$ is a rule which assigns to each graph $\cs\in \gr$ an element $\nu(\cs)\in k\setminus\{0\}$ such that the following hold:
		\begin{enumerate} 
			\item $\nu(\cs_1)=\nu(\cs_2)$ for all $\cs_1,\cs_2\in \gr$ with $\cs_1\cong \cs_2$, \label{condition 1 nu}
			\item $\nu(\emptyset)=1$, \label{condition 2 nu}
			\item Suppose $\cs_0,\cs_1,\cs_2\in \gr$ and $\iota_1\colon \cs_0\to \cs_1, \iota_2\colon \cs_0\to \cs_2$ are embeddings. Let $\ca_1,\ldots, \ca_r$ be representatives of isomorphism classes of all amalgamations of $(\iota_1,\iota_2)$.
			Then 
			\[
				\nu(\cs_1)\cdot \nu(\cs_2)= \nu(\cs_0)\cdot \sum_{i=1}^{r}\nu(\ca_i).
			\] \label{condition 3 nu}
		\end{enumerate}
	\end{defi}
	\begin{rema} \label{intuition for measure}
		Let $\cS$ is a finite homogeneous graph, for example the complete graph $K_m$ for $m\in \N$ or a complete multipartite graph. Assume $k$ is a field, whose characteristic does not divide the order of $\Aut(\cS)$, then $\nu(\cs)=|\cS^{[\cs]}|=|\{\cs \hookrightarrow \cS\}|$ is the unique R-measure, where $\cS^{[\cs]}=\{\iota\colon \cs\hookrightarrow \cS\}$ is the set of embeddings. The correct intuition for a measure for infinite graphs $\cS$ is that it is a well-behaved replacement for this cardinality. 
	\end{rema}
	We want to mention the following lemma, which is useful for computing R-measures for graphs.
	\begin{lemm}[One-point lemma, \cite{MeyerWojciechowski2026}] \label{technical lemma}
		Let $\nu\colon \gr\to k\setminus\{0\}$ as in Definition \ref{defi R-measure} be an assignment, which satisfies \eqref{condition 1 nu} and \eqref{condition 2 nu}, but satisfies \eqref{condition 3 nu} only for one-point amalgamations. Then $\nu$ satisfies \eqref{condition 3 nu} for all amalgamations, and in particular is an R-measure.
	\end{lemm}
	\begin{proof}
		Every amalgamation can be iteratively constructed from one-point amalgamations. The statement follows by induction. 
	\end{proof}
	
	\begin{beis} \label{example R-measure on finite totally ordered sets}
	Consider the graph $\cS=(\mQ,<)$. Then $\age(\cS)$ is the class of finite totally ordered sets. Let $k$ be any field. There is a unique R-measure $\nu$ on this class. It assigns to the finite totally ordered set $\{1<2<\cdots < n\}$ the value $\nu(\{1<2<\cdots < n\})=(-1)^n$. This fits with the intuition from Remark \ref{intuition for measure} as follows. The value $\nu(\{1\})$ of the totally ordered set $\{1\}$ is $-1$, since $\cS^{[1]}=\mQ$ and $\mu(\mQ)=-1$ in Example \ref{example measure on Q}. 
	\end{beis}
	We recall an important theorem from \cite{harmansnowden22}, which gives the precise connection between regular measures on $\Aut(\cS)$ on R-measures on $\age(\cS)$.
	\begin{theo}[{\cite[Corollary 6.10]{harmansnowden22}}] \label{theorem equivalence of measures}
		Assume $\operatorname{char}(k)=0$. Let $G=\Aut(\cS)$. Moreover assume that every open subgroup $U\subseteq G$ contains $U(x_1,\ldots, x_n)$ for a finite set $\{x_1,\ldots, x_n\}\subseteq \cS$ such that the index $[U\colon U(x_1,\ldots, x_n)]$ is finite.
		Then there is a one-to-one correspondence:
		\[
			\{
				\text{Regular measures $\mu$ on $G$}\} \xleftrightarrow[]{1\colon 1} \{ \text{R-measures $\nu$ on $\age(\cS)$}
			\}.
		\]		
	\end{theo}
	\begin{proof}
		We explain how the correspondence works. The full proof can be found in \cite[\S1 - \S6]{harmansnowden22} culminating in \cite[Corollary 6.10]{harmansnowden22}.
		
		A regular measure $\mu$ on $G$ is sent to the R-measure $\nu_{\mu}$, which assigns to each finite structure $\cs$ of cardinality $n$ the measure $\nu_\mu(\cs)\coloneqq \mu(\cS^{[\cs]})$. Here $\cS^{[\cs]}\coloneqq\{\iota\colon \cs\hookrightarrow \cS\}$ denotes the set of all embeddings of $\cs$ into $\cS$ viewed as $G$-set.
		
		A R-measure $\nu$ is sent to the regular measure $\mu_{\nu}$, which satisfies the following property. Let $(x_1,\ldots, x_n)\in X^{(n)}$ for some $n\in \mN$. We consider subgraph $\cs$ of $\cS$ with these vertices. Let $m\in \mN$ with $m\leq n$, and consider the subgraph $\cs'\subseteq \cs$ whose vertices are $x_1,\ldots, x_m$. One assigns to the $\hat{G}$-set $U(x_1,\ldots, x_m)/U(x_1,\ldots, x_n)$ the value
		\[
		\mu_{\nu}(U(x_1,\ldots, x_m)/U(x_1,\ldots, x_n))\coloneqq \frac{\nu(\cs)}{\nu(\cs')}.
		\]
		For an arbitrary inclusion $U\subseteq U'$ of open subgroups 
		one chooses $x_1,\ldots, x_n\in \cS$ such that $U(x_1,\ldots, x_n)\subseteq U$ has finite index, and $x_1',\ldots x_m'\in \cS$ such that $U(x_1',\ldots, x_m')\subseteq U'$ has finite index. We denote by $\cs$ the subgraph with vertices $\{x_1,\ldots, x_n\}$
		and by $\cs'$ the subgraph with vertices $\{x_1',\ldots, x_m'\}$.
		Then one assigns to the transitive $U'$-set $U'/U$
		\[
		\mu_{\nu}(U'/U))\coloneqq\frac{[U'\colon U(x_1',\ldots, x_n')]\cdot \nu(\cs)}{[U\colon U(x_1,\ldots, x_m)]\cdot \nu(\cs')}.
		\]
		This assignment is additively completed to disjoint unions to obtain a value $\mu_{\nu}(Y)$ for each smooth and orbit-finite $\hat{G}$-set $Y$.
	\end{proof}
	\begin{beis}[{\cite[\S14.6]{harmansnowden22}}] \label{example measure for S}
		We explain the measure from Example \ref{example measure for S_infty} in more detail. Assume $\operatorname{char}(k)=0$. Let $\lambda\in k$. There is a measure $\mu_{\lambda}$ on the oligomorphic group $S_\infty=\Sym(\N)$, which assigns to the $S_{\infty}$-set $\N$ the value $\mu_{\lambda}(\N)=\lambda$. 
		One can show that for every open subgroup $U$ there exists $n\in \mN$, such that $U$ is conjugate to $H\times U(1,\ldots, n)$ for some finite subgroup $H\subseteq S_n$. Given another open subgroup $U'$ with $U\subseteq U'$, one finds such $m\in \mN$ with $m\leq n$, such that $U'$ is conjugate to $H'\times U(1,\ldots, m)$ where $H'\subseteq S_m$ is finite. One then assigns to $U'/U$ the value
		\[
			\mu_{\lambda}(U'/U)=\frac{|H'|(\lambda-m) (\lambda-m-1)\cdots (\lambda-n+1)}{|H|},
		\]
		with $\mu_{\lambda}(U'/U)=\frac{|H'|}{|H|}$ for $n=m$.
		If $\mu_{\lambda}$ is regular, it comes from an R-measure $\nu_{\lambda}$ on the class of finite sets. This R-measure assigns to the finite set $\{1,\ldots, n\}$ the value $\lambda(\lambda-1)\cdots (\lambda-n+1)$. 
		We finish this example with some explicit values. 
		Let $n\in \mN$ then
		\[
			\mu_{\lambda}(\N\setminus\{1\})=\lambda-1, \quad \mu_{\lambda}(\N^n)=\lambda^n, \quad \mu_{\lambda}(\N^{(n)})=\lambda(\lambda-1)\cdots (\lambda-n+1)
		\]
		and $\mu_{\lambda}\left(\binom{\N}{n}\right)=\frac{\lambda\cdot (\lambda-1)\cdots (\lambda-n+1)}{n!}$.
	\end{beis}
	\begin{rema}
		Note that the proof of Theorem \ref{theorem equivalence of measures} only requires that the characteristic of $k$ is non-zero, if there are open subgroups, which are not of the form $U(x_1,\ldots, x_n)$ for some $n\in \mN$ and $x_1,\ldots, x_n\in \cS$. This is the reason why the R-measure $\nu$ on finite totally ordered sets from Example \ref{example R-measure on finite totally ordered sets} induces a regular measure on $\Aut(\mQ, <)$ for any field. Harman and Snowden use this in \cite{harmansnowden24delannoy} to introduce the Delannoy category.
	\end{rema}
	\section{Representation-theoretic Preliminaries} \label{section linearizing}
	We fix a field $k$. 
 	\subsection{Linearizing without Measures} \label{subsection linearizing without measures}
 	 Throughout this section $G$ is any group.
 	 \begin{defi}
 	 	We define $\Rep(G)$ to be the category of all $k$-linear \textit{representations} of $G$, i.e.\ $k$-vector spaces with a linear $G$-action.
 	 	Let $V,W\in \Rep(G)$. We denote by
 	 	\begin{enumerate}[label=\roman*), leftmargin=1cm]
 	 		\item $\Hom_G(V,W)\subseteq \Hom_k(V,W)$ the vector space of all \textit{$G$-morphisms},
 	 		\item $V^{G}=\{v\in V\mid gv=v \text{ for all } g\in G\}\subseteq V$ the \textit{subspace of $G$-invariants},
 	 		\item $k_{\operatorname{triv}}$ the $1$-dimensional trivial representation of $G$,
 	 		\item $V\otimes_k W$ the \textit{tensor product} with the usual $G$-action given on pure tensors by $g\cdot v\otimes w=gv\otimes gw$, where $g\in G$, $v\in V$, $w\in W$.
 	 	\end{enumerate}
 	 \end{defi}
 	 \begin{defi}
 	 	Let $Y$ be a $G$-set. Consider the set $\Maps(Y,k)$ of all maps $v\colon Y\to k$. We view $\Maps(Y,k)$ as a $G$-representation via the action $(g\cdot v)(y)\coloneqq v(g^{-1}y)$, where $g\in G$, $v\in \Maps(Y,k)$, $y\in Y$.
 	 	%Given a $G$-orbit $\cO\in G\backslash Y$ we set %$v_{\cO}=\sum_{g\in G}^{}$
 	 \end{defi}
 	 
 	 	Classically one would never consider the entire space of functions $Y\to k$ for infinite $G$-sets. Instead one would consider some well-behaved subspace of $\Maps(Y,k)$. The naive choice is the following:
  	\begin{defi}
  		Let $Y$ be a $G$-set. We define the \textit{naive linearization} $kY$ of $Y$ to be the $G$-subrepresentation
  		\[
  		kY\coloneqq\{v\mid v(y)=0 \text{ for all but finitely many $y\in Y$}\}\subseteq \Maps(Y,k).
  		\]
  		We denote for an element $y\in Y$ by $v_y\colon y'\mapsto \delta_{y,y'}$ the \textit{characteristic function}.
  		By definition of $kY$ the set $\{v_y \mid y\in Y\}$ forms a $k$-basis of $kY$, which we call the \textit{standard basis}. 
  	\end{defi}
	The following lemma explains the connection between linearizations and orbits.
	\begin{lemm} \label{lemma linearization}
		Let $Y$ be a $G$-set. The space $\Maps(Y,k)^{G}$ of invariants has a vector space basis labeled by all $G$-orbits
		\[
			\left\{v_{\cO} \mid \cO\in G\backslash Y\right\}, \quad \text{where } \,  v_{\cO}(y)\coloneqq \begin{cases}
				1 & \text{if $y\in \cO$}	\\
				0 & \text{otherwise. }
			\end{cases} \text{for $y\in Y$.} 
		\]
		In contrast the space $(kY)^{G}$ has a basis labeled by finite $G$-orbits
		\[
			\{
			v_{\cO}=\sum_{y\in \cO}v_y \mid \cO\in G\backslash Y, \text{ $\cO$ is finite}.
			\}
		\]
	\end{lemm}
	\begin{proof}
		This lemma is a special case of Lemma \ref{lemma missing homomorphisms}. We explain how to deduce it as part of Example \ref{example: All the orbits!}.
	\end{proof}
	Lemma \ref{lemma linearization} says that from an orbit counting point of view one should take $\Maps(Y,k)$ as linearization of $Y$ and not $kY$ to remember infinite $G$-orbits. The space of all maps is way to large: its dimension is uncountable even for countably infinite $Y$. We will explain how \cite{harmansnowden22} solve this problem in Subsection \ref{subsection linearizing with measures} by introducing the Schwartz space $\cC(Y)$, which sits in between $kY$ and $\Maps(Y,k)$.
 	\begin{lemm} \label{lemma linearizing product}
 		Let $Y_1$, $Y_2$ be $G$-sets. 
 		There are natural isomorphisms
 		\[
 		kY_1\oplus kY_2\cong k(Y_1\sqcup Y_2), \quad  kY_1 \otimes kY_2\cong k(Y_1\times Y_2).
 		\]
 	\end{lemm}
	We finally want to recall an important definition.
	\begin{defi}[{\cite[\S5]{cameron90}}] \label{defi orbit algebra}
		Let $G\subseteq \Sym(X)$ be an oligomorphic permutation group.
		Let $Y$ be a smooth and orbit-finite $G$-set. We define \textit{the orbit algebra} as the vector space
		\[
			\H_{G,Y}^{\star}\coloneqq \bigoplus_{n\geq 0}H^n_{G,Y} \quad \text{ with } \quad H^n_{G,Y}\coloneqq \Maps\left(\binom{Y}{n}, \C \right)^{G} \! \! \!.
		\]
		For each $m,n\in \mN$ we define a product $\cdot\colon \H_{G,Y}^{m}\otimes_\C \H_{G,Y}^{n}\to \H_{G,Y}^{n+m}$ by setting for $S\in \binom{Y}{m+n}$
		\[
			(v_1\cdot v_2)(S)\coloneqq \sum_{S=S_1\sqcup S_2, |S_1|=m, |S_2|=n}v_1(S_1)v_2(S_2).
		\]
		This turns $\H_{G,Y}^{\star}$ into a commutative, graded algebra.
	\end{defi}
	
 \subsection{Linearizing with Measures} \label{subsection linearizing with measures}
We fix a set $X$ and an oligomorphic group $G\subseteq \Aut(X)$. This section is about Schwartz spaces introduced in \cite{harmansnowden22}. Lemma \ref{lemma missing homomorphisms} explains how these fix a flaw of naive linearizations of $G$-sets, which don't contain all the morphisms/invariants, which one would expect coming from the representation theory of finite groups.

 \begin{defi}
 	Let $Y$ be a smooth (but not necessarily orbit-finite) $G$-set. 
 	The \textit{Schwartz space} $\cC(Y)$ is defined to be the subset of $\Maps(Y,k)$ given by
 	\[
 	\cC(Y)\coloneqq \{ f \mid \text{$f$ is $U$-invariant for an open subgroup $U\subseteq G$}\}.
 	\]
 	Here a function $f\colon Y\to k$ is called \textit{ $U$-invariant}, if it is fixed by the action of $U$ on $\Maps(Y,k)$.
 \end{defi}
	\begin{lemm}
		We have inclusions of subrepresentations
		\[
			kY\subseteq \cC(Y)\subseteq \Maps(Y,k).
		\]
	\end{lemm}
\begin{proof}
	First note that the Schwartz space is a subspace: It is closed under scalar multiplication, and the sum of two vectors is invariant under the intersection of two open subgroups.
	It is also a subrepresentation. Indeed let $f\in \cC(Y)$ be $U$-invariant for some open subgroup $U\subseteq G$, and let $g\in G$. Then $gf$ is invariant under $gUg^{-1}$, which is again an open subgroup of $G$. The naive linearization $kY$ is a subspace of the Schwartz space $\cC(X)$, since $v_y$ is fixed by the point stabilizer $G_y$, which is open since $Y$ is smooth.
\end{proof}
 \begin{beis} \label{example Schwartz space}
 	Let $G=S_{\infty}=\Sym(\N)$ and consider $Y=\N$. Then 
 	\[
 		k\N=\spann \{v_n \mid n\in \N\}\subsetneq \cC(\N)=\spann \{v_n \mid n\in \N\} \oplus  \spann \{\sum_{n\in \N}v_n\}.
  	\] 
  	Here $\sum_{n\in \N}v_n$ stands for the $S_{\infty}$-invariant function $v$ with $v(n)=1$ for all $n\in \N$.
 \end{beis}
\begin{rema}
	The price one has to pay for working with Schwartz spaces is Lemma \ref{lemma linearizing product}. Let $Y$, $Y'$ be smooth, orbit-finite $G$-sets. We have $\cC(Y\sqcup Y')\cong \cC(Y)\oplus \cC(Y')$. There also exists a monomorphism $\cC(Y)\otimes_k \cC(Y')\hookrightarrow \cC(Y\times Y')$, but it is general not an isomorphism. 
%	For instance $\cC(\N)^{\otimes 2}\ncong \cC(\N^{2})$ as representation of $S_{\infty}$. Hence Lemma \ref{lemma linearizing product} does not carry over in full generality.
\end{rema}
 We will fix some notation for possibly infinite sums as in Example \ref{example Schwartz space}. 
 \begin{defi} \label{notation infinite sums}
 	Let $Y$ be a smooth $G$-set. Let $v\in \cC(Y)$. We define the coefficient $a_y\coloneqq v(y)$ for $y\in Y$.
 	We will from now on write $v=\sum_{y\in Y}a_yv_y$ for $v$ ignoring whether $v\in kY$ (i.e.\ whether this is an actual sum of basis vectors). 
 \end{defi} 
 \begin{beis} \label{defi finite decomposition}
 	Let $Y$ be a smooth $G$-set. Let $v=\sum_{y\in Y}a_y v_y\in \cC(Y)$. Choose an open subgroup $U$ which fixes $v$. Then we can decompose $Y$ into finitely many $U$-orbits $\cO$ and write $a_{\cO}=a_y$ if $y\in \cO$ (this is independent of choice of $y$ since $v$ is fixed by $U$). Then $v$ can be written as a finite sum of `infinite sums'
 	\[
 		v=\sum_{\cO\in U\backslash Y}a_{\cO} \sum_{y\in \cO} v_y.
 	\]	
 	%We call this a \textit{finite decomposition} of $v$.
 \end{beis}
We fix smooth and orbit-finite $G$-sets $Y$ and $Y'$. 
\begin{defi} 
	Let $\cO\in G\backslash (Y'\times Y)$ be an orbit. We say $\cO$ is \textit{$(Y\to Y')$-small}, if it satisfies the following condition: 
	\begin{equation} \label{condition finiteness}
		\text{For all $y\in Y$ the set $\{y'\in Y'\mid (y',y)\in \cO\}$ is finite.}
	\end{equation}
\end{defi}
\begin{lemm} \label{lemma missing homomorphisms}
	We have $\dim_k\Hom_G(kY, \cC(Y'))= |G\backslash (Y'\times Y)|$
	and
	\[
	\dim_k\Hom_G(kY, kY')=|\{\cO\in  G\backslash (Y'\times Y) \mid \text{$\cO$ is $(Y\to Y')$-small} \}|.
	\]
	A $k$-basis of $\Hom_G(kY, \cC(Y))$ is given by the \textit{indicator functions} $\delta_\cO$ of orbits $\cO\in G\backslash (Y'\times Y)$. Here $\delta_{\cO}$ is defined on the basis of $kY$ via
	\[
	\delta_\cO(v_y)=\sum_{y'\in Y\colon (y',y)\in \cO}v_{y'}\in \cC(Y').
	\]
\end{lemm}
\begin{proof}
	An element $f\in \Hom_G(kY,\cC(Y'))$ is determined by the images $f(v_y)$ of the standard basis elements $v_y$ where $y\in Y$. Let $y\in Y$ and write $f(v_y)=\sum_{y'\in Y'} a_{y',y} v_{y'}$
	for some coefficients $a_{y',y}\in k$, utilizing the notation for infinite sums in Definition \ref{notation infinite sums}. The condition $g\cdot f(v_y)=f(g\cdot v_y)=f(v_{g y})$ for all $g\in G$ is equivalent to the condition $a_{y',gy}=a_{g^{-1}y',y}$, which is equivalent to $a_{gy',gy}=a_{y',y}$ for all $g\in G$ and $(y',y)\in Y'\times Y$. Hence $G$-equivariance is equivalent to the condition that the coefficients are constant along $G$-orbits on $Y'\times Y$.
	We can consider the $G$-subset $\{(y',y)\mid a_{y',y}\neq 0\}\subseteq Y'\times Y$ and decompose it into finitely many $G$-orbits $\cO_1,\ldots, \cO_r$ by Lemma \ref{lemma products of smooth}. We fix representatives $(y_i',y_i)$ of each $\cO_i$. By the above consideration we have a unique decomposition $f=\sum_{i=1}^{r}a_{y_i',y_i}\delta_{\cO_i}$. Hence the indicator functions of $G$-orbits give a basis of $\Hom_G(kY,\cC(Y'))$. By definition of $kY'$ the coefficient $a_{y',y}$ is non-zero for only finitely many $
	y'\in Y'$ once $y$ is fixed. This finiteness allows in the decomposition only for those $G$-orbits $\cO_i$ which satisfy condition \ref{condition finiteness}, i.e.\ $(Y\to Y')$-small $G$-orbits.
\end{proof}
\begin{beis} \label{example: All the orbits!}
	We explain how Lemma \ref{lemma missing homomorphisms} implies Lemma \ref{lemma linearization}. Consider $Y=\{\operatorname{pt}\}$ and $Y'=X$. Then $kY=k_{\operatorname{triv}}$ is the trivial representation. For a $G$-representation $V$ one can identify $\Hom_G(k_{\operatorname{triv}}, V)$ with the subspace $V^{G}\subseteq V$ of $G$-invariants. Lemma \ref{lemma missing homomorphisms} says that $(kX)^G$ has a basis labeled by $(\{\operatorname{pt}\}\to X)$-small $G$-orbits in $X\times \{\operatorname{pt}\}$, which are precisely finite $G$-orbits in $X$. In contrast $\cC(X)^{G}$ has a basis labeled by all $G$-orbits. Hence, the Schwartz space $\cC(X)$ should be imagined as a nice completion of $kX$.
\end{beis}
  The problem is that morphisms $f\colon kY\to \cC(Y')$ and $g\colon kY'\to \cC(Y'')$ can not be naively composed. For the rest of the section we fix a measure $\mu$ on $G$ with values in $k$. We will use $\mu$ to extend morphisms $f\colon kY\to \cC(Y')$ to morphisms $f_{\mu}\colon \cC(Y)\to \cC(Y')$. 
 \begin{defi} \label{definition f mu}
 	Let $f\in \Hom_G(kY, \cC(Y'))$ and fix coefficients $a_{y',y}$ such that $f(v_y)=\sum_{y'\in Y'}{a_{y', y}}v_{y'}$ for $y\in Y$. 
 	We define its \textit{$\mu$-lift}
 	$f_{\mu}\colon \cC(Y)\to \cC(Y')$ in the following way. Let $v=\sum_{y\in Y}\lambda_y v_y\in \cC(Y)$ and choose an open subgroup $U\subseteq G$ such that $v$ is $U$-invariant. Let $y'\in Y'$. We consider the $(U(y')\cap U)$-subset 
 	\[
 		\operatorname{Supp}_{f,y',v}\coloneqq \{y\in Y \mid a_{y', y}\neq 0, \, \lambda_y\neq 0\}\subseteq Y,
 	\]
 	which is smooth and orbit-finite, since $Y$ is. We decompose it into $(U(y')\cap U)$-orbits 
 	\[
 		\operatorname{Supp}_{f,y',v}=\cO_{y',v,1}\sqcup\cO_{y',v,2} \sqcup \cdots \sqcup \cO_{y',v,r(y')} 
 	\]
 	for some $r(y')\in \N$. We choose representatives $y_{y',v,i}\in \cO_{y',v,i}$ of these orbits. We set
 	\[
 		f_\mu(v)\coloneqq \sum_{y'\in Y'} \sum_{i=1}^{r(y')}\mu(\cO_{y',v,i})a_{y',y_{y',v,i}}\lambda_{y_{y',v,i}}v_{y'}\in \cC(Y').
 	\]
 \end{defi}
 \begin{beis}
 	This example is dual to Example \ref{example: All the orbits!}. Every $G$-orbit on $X$ can be viewed as an orbit $\cO$ on $\{\operatorname{pt}\}\times X$ and hence gives a characteristic function $\delta_{\cO}\colon kX\to \cC({\pt})=k(\{\pt\})=k_{\operatorname{triv}}$, see Lemma \ref{lemma missing homomorphisms}. Consider $\sum_{x\in \cO}v_x\in \cC(X)$. Then $(\delta_{\cO})_{_\mu}(\sum_{x\in \cO}v_x)=\mu(\cO)\cdot 1\in k_{\operatorname{triv}}$.
 \end{beis}
 The following technical proposition is crucial.
 \begin{prop} \label{properties f mu}
 	Let $f$ be as in Definition \ref{definition f mu}. The following statements hold about $\mu$-lifts:
 	\begin{enumerate}
 		\item $f_{\mu}$ is well-defined.
 		\item \label{property ii} $f_{\mu}$ is a morphism of $G$-representations.
 		\item $\restr{f_{\mu}}{kY}=f$.
 		\item Consider the inclusion $\iota\colon kY\to \cC(Y)$. Then $\iota_{\mu}=\id_{\cC(Y)}$.
 		\item \label{property v} Given $G$-morphisms $f\colon kY\to \cC(Y')$, $g\colon kY'\to \cC(Y'')$, then $g_{\mu}\circ f_{\mu}=(g_{\mu}\circ f)_{\mu}$. 
 	\end{enumerate} 
 \end{prop}
 \begin{proof}
 	The map $f_{\mu}$ is different notation for $A\colon \operatorname{Vec}_X\to \operatorname{Vec}_Y$ corresponding to a $G$-invariant $Y'\times Y$-matrix in \cite[\S7.2]{harmansnowden22}. The properties follow by translating properties of matrices through the forgetful functor $\Phi_0$ in \cite[Proposition 10.13]{harmansnowden22}, see \cite[\S8.2]{harmansnowden22}. The proof heavily relies on the axioms of a measure. 
 \end{proof}
 \begin{rema}
 	Property \ref{property v} in Proposition \ref{properties f mu} says that a composition of $\mu$-lifts $g_\mu\circ f_\mu\colon \cC(Y)\to \cC(Y'')$ is itself a $\mu$-lift. In particular it is determined by the values of the basis vectors $v_y\in kY$ where $y\in Y$.
 \end{rema}
 \begin{rema} \label{remark Perm category for Harman Snowden}
 	  There is a (in general not full) subcategory $\Perm$ of $\Rep(G)$, whose objects are Schwartz spaces $\cC(Y)$ of smooth and orbit-finite $G$-sets $Y$ and whose morphisms $\cC(Y)\to \cC(Y')$ are those which are $\mu$-lifts of morphisms $kY\to \cC(Y')$. This category is the main object of study in \cite{harmansnowden22}. The abbreviation stands for `category of permutation modules'.
 \end{rema}

%\begin{lemm}
%	The identity morphism $\id_{\cC(Y)}$ is continuous, since it is the closure of the inclusion $kY\hookrightarrow \cC(Y)$. The composition of continuous homomorphisms between Schwartz spaces is smooth. 
%\end{lemm}
	\section{Lie Algebra Actions for Oligomorphic Groups} \label{section Lie algebras}
	We fix a field $k$, a set $X$, an oligomorphic permutation group $G\subseteq \Sym(X)$, and a measure $\mu$ on $G$ with values in $k$.
	\subsection{An Infinite Tensor Power} \label{subsection infinite tensor power} \label{subsection: sl2 and Stanley}
We fix $r\in \N$. We start by recalling some notation for the general linear Lie algebra.
\begin{defi}
	We consider the $k$-algebra $(\M_{r \times r}(k), \cdot)$ of $r\times r$ matrices. For $1\leq i,j\leq  r$ we denote by $E_{ij}\in \M_{r \times r}(k)$ the \textit{elementary matrix}, whose only non-zero entry is the $i$-$j$-th entry, which is $1$. We denote by $[\blank, \blank]$ the \textit{Lie bracket} $[a,b]=ab-ba$ for $a,b\in \M_{r \times r}(k)$. The tuple $(\M_{r \times r}(k),[\blank, \blank])$ is the \textit{general linear Lie algebra} denoted $\mathfrak{gl}_r(k)$. In the special case $r=2$ we will consider the \textit{special linear Lie algebra} $\sl_2(k)\subseteq \gl_2(k)$ consisting of trace $0$ matrices. We set
	\[
		e\coloneqq E_{12}= \begin{psmallmatrix}
			0 & 1 \\ 0 & 0
		\end{psmallmatrix}, \quad f\coloneqq E_{21}=\begin{psmallmatrix}
		0 & 0 \\ 1 & 0
		\end{psmallmatrix}, \quad h\coloneqq E_{11}-E_{22} =\begin{psmallmatrix}
		1 & 0 \\ 0 & -1
		\end{psmallmatrix}.
	\] 
\end{defi}
We recall formulas for the Lie bracket of $\gl_r(k)$ for later reference.
\begin{lemm} \label{lemma defining relations of glr}
	Let $i,j,l,m\in \{1,\ldots, r\}$. We have 
	\begin{equation} \label{gl_r-relations}
	[E_{ij}, E_{lm}]=\begin{cases}
		E_{ii}-E_{jj}, & \text{if $i=m$, $j=l$,} \\
		E_{im}, & \text{if $i\neq m$, $j=l$,}\\
		-E_{lj}, & \text{if $i=m$, $j\neq l$,}\\
		0, &\text{if $i\neq m$, $j\neq m$.}
	\end{cases}
	\end{equation}
	In particular an action of $\gl_r(k)$ on a vector space $V$ consists of endomorphisms $E_{ij}\cdot \blank\in \End_k(V)$ for each $1\leq i,j\leq r$ such that their commutator $[E_{ij}\cdot \blank, E_{lm}\cdot \blank]$ satisfies the relations \ref{gl_r-relations}.
\end{lemm}
%\begin{proof}
%	This is a direct calculation using $E_{ij}E_{lm}=\delta_{j,l}E_{im}$. The second part is the definition of an action.
%\end{proof}
We fix a smooth and orbit finite $G$-set $Y$. Recall that $2_{\hat{G}}^{Y}\subseteq 2^Y$ denotes the set of $\hat{G}$-subsets of $Y$, cf.\ Definition \ref{definition G hat sets}.
\begin{defi} \label{definition par}
	We define the set of \textit{ordered partitions of $Y$} (into $r$ parts, which are $\hat{G}$-subsets) as
	\[
	\pa_r(Y)\coloneqq \{(Y_1,\cdots, Y_r) \mid  \bigcup_{i=1}^{r}Y_i=Y, \, Y_i\cap Y_j=\emptyset \text{ for $i\neq j$}\}\subseteq \left(2_{\hat{G}}^{Y}\right)^r.
	\]
\end{defi}
\begin{rema} \label{remark: identification with powerset r=2}
	In the case $r=2$ we will identify $\pa_2(Y)$ with $2_{\hat{G}}^{Y}$ via $(Y_1,Y_2)\mapsto Y_1$, since the complement of a $\hat{G}$-subset is a $\hat{G}$-subset by Lemma \ref{lemma elementary set-theoretic properties of hat G sets}. 
	%We will use this to simplify things in section Subsection \ref{subsection: sl2 and Stanley}.
\end{rema}
%\begin{beis}
%	Let $G=S_{\infty}$ and let $Y=\N$. Then $2_{\hat{G}}^{\N}$ consists precisely of the finite or cofinite subsets of $\N$ as in Example \ref{example finite or cofinite subsets}.
%\end{beis}
\begin{lemm}
	Let $Y$ be a smooth $G$-set. The componentwise $G$-action on $\pa_r(Y)$ is smooth.
\end{lemm}
\begin{proof}
	First note that $2_{\hat{G}}^{Y}\subseteq 2^Y$ is a $G$-subset of the power-set with the elementwise $G$-action. Indeed let $Z\subseteq Y$ is a $\hat{G}$-subset of $Y$, say $Z$ is a $U$-subset of $Y$ for some open subgroup $U\subseteq G$ and let $g\in G$. Then $gZ$ is preserved by the open subgroup $gUg^{-1}$. Hence $gZ$ is also a $\hat{G}$-subset of $Y$. The set-stabilizer of $Z$ contains $U$ by definition, and every subgroup of $G$ containing an open subgroup is open by Proposition \ref{topological properties}. Now $\left(2_{\hat{G}}^{Y}\right)^r$ is smooth by Lemma \ref{lemma products of smooth} and hence its $G$-subset $\pa_r(Y)$ is also smooth.
\end{proof}

\begin{defi} \label{definition tensor power}
	We define the \textit{$Y$-th tensor power of $k^r$} as the $G$-representation
	\[
	(k^r)^{\otimes Y}\coloneqq \cC(\pa_r(Y)).
	\]
\end{defi}
\begin{beis} \label{example: finite Y}
	If $Y$ is finite, 
	then $\pa_r(Y)$ is isomorphic as $G$-set to $\Maps(Y, \{1,\ldots, r\})$. Indeed, the $\hat{G}$-subsets are all subsets of $Y$ by Lemma \ref{lemma elementary set-theoretic properties of hat G sets}. The isomorphism identifies $(Y_1,\ldots, Y_r)$ with the map, which maps $Y_i$ to $i$. We can further identify $\Maps(Y, \{1,\ldots, r\})\cong \prod_{y\in Y}\{1,\ldots, r\}$. After linearizing products turn to tensor products by Lemma \ref{lemma linearizing product}, so that we have $k\pa_r(Y)\cong (k^r)^{\otimes |Y|}$. Note here also that $k\pa_r(Y)=\cC(\pa_r(Y))$, since $\pa_r(Y)$ is finite. In total we have
	\[
		(k^r)^{\otimes Y}=\cC(\pa_r(Y))=k\pa_r(Y)\cong k \prod_{y\in Y}\{1,\ldots, r\} \cong (k^r)^{\otimes |Y|}.
	\]
\end{beis}
\begin{defi} \label{defi: action glr}
	 Let $1\leq i,j\leq r$. We first define a $G$-equivariant linear map $E_{ij}\colon k\pa_r(Y)\to (k^r)^{\otimes Y}$ by linearly extending the assignment
	 \[
	 	E_{ij}\cdot v_{Y_1,\ldots, Y_r}=\sum_{y\in Y_j}v_{Y_1,\ldots, Y_i\cup\{y\}, \ldots Y_j\setminus\{y \}, \ldots, Y_r} \in \cC(\pa_r(Y))
	 \]
	 on standard basis vectors labeled by $(Y_1,\ldots, Y_r)\in \pa_r(Y)$.
	 We extend $E_{ij}\cdot \blank$ to an endomorphism $(E_{ij} \cdot \blank)_{\mu}\in \End_k((k^r)^{\otimes Y})$ using Definition \ref{definition f mu}. For simplicity we will write $E_{ij} \cdot \blank$ instead of $(E_{ij} \cdot \blank)_{\mu}$. By convention we interpret the operator for $i=j$ as rescaling
	 \[
	  E_{ii}\cdot v_{Y_1,\ldots, Y_r}\coloneqq \mu(Y_i)\cdot v_{Y_1,\ldots, Y_r}.
	  \]
\end{defi}
The following is the main theorem of this section.
\begin{theo} \label{theorem lie algebra action}
	The operators $E_{ij}\cdot \blank$ from Definition \ref{defi: action glr} are well-defined and give an action $\gl_r(k)\acts (k^r)^{\otimes Y}$,
	which commutes with the action of $G$. In particular we obtain a natural action on invariants
	\[
		\gl_r(k)\acts \left((k^r)^{\otimes Y}\right)^G.
	\]
\end{theo}
\begin{proof}
	The ``summands'', which the operator $E_{ij}\cdot \blank$ produces are indeed labeled by elements in $\pa_r(Y)$, since $\hat{G}$-sets are compatible with adding or removing one element by Lemma \ref{lemma elementary set-theoretic properties of hat G sets}. For $Y_1,\ldots, Y_r\in \pa_r(Y)$ there exist open subgroups $U_1,\ldots, U_r$ which act on $Y_1,\ldots, Y_r$ respectively. Hence $U\coloneqq \bigcap_{i=1}^rU_i$ acts on each $Y_j$ with $1\leq j\leq r$. The expression $\sum_{y\in Y_j}v_{Y_1,\ldots, Y_i\cup\{y\}, \ldots Y_j\setminus\{y \}, \ldots, Y_r}$ is $U$-invariant and hence contained in the Schwartz space $\cC(\pa_r(Y))$.
	The operators $E_{ij}$ commute with the action of $G$ by Proposition \ref{properties f mu}, part \eqref{property ii}. For basis vectors the $G$-equivariance holds since it does not matter whether we first act with $G$ on a partition of $Y$ and then move one element or do it the other way around. 
	
	One has to check the $\gl_r(k)$-relations from \ref{lemma defining relations of glr} hold, which can be checked on basis vectors using Proposition \ref{properties f mu}, part \eqref{property v}, by a case by case distinction. We show one of the relations, namely $[E_{ij},E_{ji}]=E_{ii}-E_{jj}$ for $i\neq j$. Since all the calculations are local we can simplify notation and assume that $r=2$, $i=1$, $j=2$. This does not change the calculation, just the number of indices we would have to write. Let $(Y_1,Y_2)\in \pa_2(Y)$, we compute the action of $e=E_{12}$ and $f=E_{21}$:
	\[
		(ef)(v_{Y_1,Y_2})=e\left(\sum_{y_1\in Y_1}v_{Y_1\setminus\{y_1\},Y_2\cup\{y_1\}}\right)=\sum_{(Y_1', Y_2')\in \pa_2(Y)} \mu(S_{Y_1,Y_1'})v_{Y_1', Y_2'}
	\]
	where for fixed $(Y_1', Y_2')\in \pa_2(Y)$ we set
	\[
		S_{Y_1,Y_1'}\coloneqq\left\{y_1\in Y_1 \mid \text{There exists } \tilde{y}_2\in Y_2\cup\{y_1\}\colon Y_1'=Y_1\setminus\{y_1\}\cup \{\tilde{y}_2\}\right\}.
	\]
	This set can be reformulated as the set
	\[
		S_{Y_1',Y_2'}=\{y_1\in Y_1 \mid Y_1\setminus \{y_1\}\subseteq Y_1'\}.
	\]
	There are three cases how large this $\hat{G}$-set is. If $Y_1=Y_1'$, then $S_{Y_1, Y_1'}=Y_1=Y_1'$ and $\mu(S_{Y_1, Y_1'})=\mu(Y_1)$.  If $\emptyset\neq Y_1\cap Y_1' \neq Y_1$ then $|S_{Y_1, Y_1'}|=1$ and in particular $\mu(S_{Y_1, Y_1'})=1$. Otherwise $Y_1$ and $Y_1'$ differ in at least two elements, and this set is empty. By the same calculation
	\[
	(fe)(v_{Y_1,Y_2})=e\left(\sum_{y_2\in Y_2}v_{Y_1\cup\{y_2\},Y_2\setminus\{y_2\}}\right)=\sum_{(Y_1', Y_2')\in \pa_2(Y)} \mu(T_{Y_1,Y_1'})v_{Y_1', Y_2'}
	\]
	where for fixed $(Y_1', Y_2')\in \pa_2(Y)$ we set
	\[
	T_{Y_1,Y_1'}\coloneqq\left\{y_2\in Y_2 \mid \text{There exists } \tilde{y}_1\in Y_1\cup\{y_2\}\colon Y_2'=Y_2\setminus\{y_2\}\cup \{\tilde{y}_1\}\right\}.
	\]
	Again the measure of this set is either $\mu(Y_2)$, if $Y_2=Y_2'$, it is $1$ if $Y_1$ and $Y_1'$ differ exactly in one point, and $0$ otherwise. 
	In total we see that
	\[
		(ef-fe)(v_{Y_1,Y_2})=(\mu(Y_1)-\mu(Y_2))v_{Y_1,Y_2}=(E_{11}-E_{22})\cdot v_{Y_1,Y_2},
	\] 
	since the remaining summands, whose coefficient is $1$ cancel.
	%We include the proof of the $\sl_2(k)$-relations in the proof of \Cref{proposition for homogenous graphs} below.
\end{proof}
\begin{rema}[Symmetry] \label{rema: symmetry}
	The set $\pa_r(Y)$ comes with a natural action of the finite symmetric group $S_r$ by permuting components. This action induces an action $S_r\acts (k^r)^{\otimes Y}$, which commutes with the action of $G$. In particular it also acts on $\left((k^r)^{\otimes Y}\right)^G$. This $S_r$-action is compatible with the $\gl_r(k)$-action in the sense that for $\sigma\in S_r$ one has
	\[
		(\sigma\cdot \blank) \circ (E_{ij} \cdot \blank) \circ (\sigma^{-1}\cdot \blank)=E_{\sigma(i)\sigma(j)} \cdot \blank
	\] for $1\leq i, j \leq r$. 
	In the special case $r=2$ and the Lie algebra $\sl_2(\C)$ conjugating with the transposition $1\mapsto 2, 2\mapsto 1$ exchanges $e=E_{12}$ and $f=E_{21}$, and replaces $h=E_{11}-E_{22}$ by $-h$. 
	For finite permutation groups $G\subseteq \Sym(X)$ the $S_r$-action gives symmetry on analogues of multinomial coefficients. 
\end{rema}
The next remark explains why one would expect a Lie algebra action and no commuting group action.
\begin{warn}[No general linear group]
	Since the symmetric group $S_r$ can be viewed as permutation matrices living in $\GL_r(k)$ one might ask if the $S_r$-action from Remark \ref{rema: symmetry} can extended to an action of $\GL_r(k)$ on $(k^r)^{\otimes Y}$.
	In the setting of finite groups $G$ acting on finite sets $Y$, there is indeed this natural action $\operatorname{GL}_r(k)$ on $(k^r)^{\otimes |Y|}\cong k\Maps(Y,\{1,\ldots, r\})$,
	cf.\ Example \ref{example: finite Y}.
	 Here $A\in \GL_r(k)$ acts on any pure tensor $v_1\otimes v_2 \otimes \cdots \otimes v_{|Y|}$ via
		\[
			A\cdot (v_1\otimes v_2 \cdots \otimes v_{|Y|})=Av_1\otimes Av_2\otimes \cdots \otimes Av_{|Y|}.
		\]
	An analogue of this action does not exist for infinite sets $Y$.
	Up to integrating (and subtleties, which we sweep under the rug) finite-dimensional representations of $\GL_r(k)$ defined in terms of polynomials correspond bijectively to representations of $\gl_r(k)$. However there are many infinite-dimensional representations of $\gl_r(k)$ with no corresponding $\GL_r(k)$-action. Polynomial representations of $\GL_r(k)$ are comodules over the coordinate ring $k[\GL_r]$, and hence unions of their finite-dimensional subrepresentations. This does not hold for the $\gl_r(k)$-representations like Verma modules, which we construct.
\end{warn}
\begin{defi}
	We define the $\hat{G}$-subset $\operatorname{Part}_r^{\operatorname{fin}}(Y)\subseteq \pa_r(Y)$ of all those set partitions $(Y_1,Y_2,\ldots, Y_{r-1},Y_r)$ such that $Y_1,\ldots, Y_{r-1}$ are finite. We define $(k^r)_{\operatorname{fin}}^{\otimes X}\coloneqq \cC(\operatorname{Part}_r^{\operatorname{fin}}(Y))\subseteq (k^r)^{\otimes Y}$.
\end{defi}

\begin{beis}
	Consider the case $r=2$. If we use Remark \ref{remark: identification with powerset r=2} and identify $\pa_2(Y)$ with $2_{\hat{G}}^Y$, then $\operatorname{Part}_2^{\operatorname{fin}}(Y)$ gets identified with $2_{\operatorname{fin}}^Y$, the set of finite subsets of $Y$. In this case $2_{\operatorname{fin}}^Y=\bigsqcup_{n\in \mN} \binom{Y}{n}$ as $G$-sets. 
\end{beis}
\begin{lemm} \label{lemma action on finite stuff}
	The actions of $G$ and $\gl_r(k)$ on $(k^r)^{\otimes Y}$ restrict to $(k^r)_{\operatorname{fin}}^{\otimes X}$. In particular we obtain an action 
	\[
		\gl_r(k)\acts \left((k^r)_{\operatorname{fin}}^{\otimes Y}\right)^G.
	\]
\end{lemm}
\begin{proof}
	For each $1\leq i,j\leq r$ the action $E_{ij}\cdot \blank$ either rescales or moves at most one element at a time per summand. 
\end{proof}
\begin{rema}
	In the case $k=\C$, $r=2$ Lemma \ref{lemma action on finite stuff} gives an action $\gl_2(\C)\acts \H_{G,Y}^{\star}$, since $\H_{G,Y}^{n}=\cC\left(\binom{Y}{n}\right)^G=\Maps\left(\binom{Y}{n}, \C\right)^G$. We will usually just consider the action of $\sl_2(\C)\subseteq \gl_2(\C)$.
\end{rema}
The following example appeared already in work of Entova-Aizenbud \cite{aizenbud2015} in the context of Deligne's interpolation category $\operatorname{\underline{Rep}}(S_t)$ (cf.\ \cite{deligne07}). This category itself heavily inspired \cite{harmansnowden22}.
\begin{beis} \label{example sl2-vermas}
	We go back to the setting of Example \ref{example: polynomial ring generating function}. There is a unique $S_{\infty}$-orbit on $\binom{\N}{n}$ for each $n\in \mN$, which gives the sequence $1,1,1,1,\ldots$. Let $k=\C$ and let $\lambda\in \C$. The space $(\C^{2})_{\operatorname{fin}}^{\otimes \N}$ has basis
	\[
	v_{[n]}\coloneqq \sum_{S=\{s_1,\ldots, s_n\}\in \binom{\N}{n}}v_{\{s_1,\ldots, s_n\}}, \quad  \text{ where } n\in \mN.
	\]
	Here $[n]$ stands for the isomorphism class of an $n$-element set. The action of $e=\begin{psmallmatrix}
		0 & 1 \\ 0 & 0
	\end{psmallmatrix}$ takes $v_{[n]}$ to $(n+1)\cdot v_{[n+1]}$, since there are $(n+1)$ embeddings of $[n]$ into $[n+1]$. Next we consider the action of $f=\begin{psmallmatrix}
		0 & 0 \\ 1 & 0
	\end{psmallmatrix}$. Fix $S\in \binom{\N}{n}$ and $S'\in \binom{\N}{n+1}$. The vector $v_{S'}$ is mapped under $f\cdot \blank $ onto a sum which contains $v_S$ as a summand if and only if $S'$ is of the form $S'=S\sqcup \{s\}$ for some $s\in \N\setminus S$. The measure $\mu_{\lambda}$ assigns to $\N\setminus S$ the value $\lambda-n$. Since this argument holds for any $S$ and $S'$ as above we have $f\cdot v_{[n+1]}=(\lambda-n)\cdot v_{[n]}$. Finally
	\[
	h\cdot v_{[n]}=(E_{11}-E_{22})v_{[n]}=\mu(\{1,\ldots, n\})-\mu(\N\setminus \{1,\ldots, n\})=(2n-\lambda)v_{[n]}.
	\]
	To summarize we obtain the following picture
	\begin{equation} \label{Verma module picture}
		\begin{tikzcd}
			{v_{[0]}} \arrow[r, red!60!black, "1", bend left] \arrow["-\lambda"', teal!60!black, loop, distance=2em, in=305, out=235] & {v_{[1]}} \arrow[l, blue!60!black,  "\lambda", bend left] \arrow[r, red!60!black, "2", bend left] \arrow["-\lambda+2"', teal!60!black, loop, distance=2em, in=305, out=235] & {v_{[2]}} \arrow[l, blue!60!black, "\lambda-1", bend left] \arrow[r, red!60!black, "3", bend left] \arrow["-\lambda+4"', teal!60!black, loop, distance=2em, in=305, out=235] & {v_{[3]}}  \arrow[l, blue!60!black,  "\lambda-2", bend left] \arrow[r, red!60!black, "4", bend left] \arrow["-\lambda+6"', teal!60!black, loop, distance=2em, in=305, out=235] & \cdots \arrow[l, blue!60!black, "\lambda-3", bend left]
		\end{tikzcd}
	\end{equation}
	The arrows pointing \textcolor{red!60!black}{right} depict the action of $\textcolor{red!60!black}{e}$, the ones \textcolor{blue!60!black}{left} the action of $\textcolor{blue!60!black}{f}$ (note that $v_{[0]}$ is annihilated by $f$) and the \textcolor{teal!60!black}{loops} the action of $\textcolor{teal!60!black}{h}$. This is precisely a lowest weight Verma module $M^{-}(-\lambda)$ for $\sl_2(\C)$, see e.g.\ \cite[\S3]{mazorchuk10}. 
\end{beis}
\begin{rema}
	In Example \ref{example sl2-vermas} there is also a subspace $(\C^{2})_{\operatorname{cofin}}^{\otimes \N}\subseteq (\C^{2})^{\otimes \N}$. Moreover $(\C^{2})^{\otimes \N}=(\C^{2})_{\operatorname{fin}}^{\otimes \N}\oplus (\C^{2})_{\operatorname{cofin}}^{\otimes \N}$.
	One can check that $\left((\C^{2})_{\operatorname{cofin}}^{\otimes \N}\right)^{S_{\infty}}$ is the usual highest weight Verma module $M(\lambda)$ of $\sl_2(\C)$. The two subspaces $\left((\C^{2})_{\operatorname{fin}}^{\otimes \N}\right)^{S_{\infty}}=M^{-}(-\lambda)$ and $\left((\C^{2})_{\operatorname{cofin}}^{\otimes \N}\right)^{S_{\infty}}=M(\lambda)$ are exchanged by the action of $S_2$ from Remark \ref{rema: symmetry}. The actions of $\sl_2(\C)$ are related via the Chevalley involution $e\mapsto f$, $f\mapsto e$, $h\mapsto -h$.
\end{rema}
The following remark explains the connection to Deligne's interpolation category $\underline{\operatorname{Rep}}(S_t)$ from \cite{deligne07}.
\begin{rema}[{\cite[Proposition 6.3.2]{aizenbud2015}}]
	In Example \ref{example sl2-vermas} the lowest weight Verma module $M^{-}(-\lambda)$ is not irreducible if and only if $\lambda=n\in \mN$, which is precisely the case if the measure $\mu_{\lambda}$ is not regular. In the interpolation category setting (write $t=\lambda$), this is precisely the case when there is a functor $\Phi_n\colon \operatorname{\underline{Rep}}(S_{t=n})\to \operatorname{Rep}_{\C}(S_n)$. Technically as pointed out in \cite{aizenbud2015} $\H_{S_{\infty},\N}=(\C^2)_{\operatorname{fin}}^{\otimes \N}$ is not contained in $\operatorname{\underline{Rep}}(S_t)$, but in an (ind-)completion, since one considers the infinite direct sum of objects for $\binom{\N}{n}$.  The completion of the functor $\Phi_n$ sends $(\C^2)_{\operatorname{fin}}^{\otimes \N}$ to $(\C^{2})^{\otimes n}$ and $M^{-}(-n)=\left( (\C^2)_{\operatorname{fin}}^{\otimes \N}\right)^{S_\infty}$ to the quotient $\left((\C^2)^{\otimes n}\right)^{S_n}$, which is precisely the $(n+1)$-dimensional irreducible representation of $\sl_2(\C)$.
\end{rema}
\begin{theo} \label{theorem: Stanley's argument}
	Let $G\subseteq \Sym(X)$ be an oligomorphic group with measure $\mu$ taking values in $\C$. Let $Y$ be an infinite, smooth, and orbit-finite $G$-set.
	The orbit algebra $\H_{G,Y}^{*}$
	has a ascending filtration by $\sl_2(\C)$-subrepresentations 
	\[
		0=\cF_0\subseteq \cF_1 \subseteq \cF_2 \subseteq \cdots \subseteq \cF_n \subseteq \cdots \subseteq \H_{G,Y}^{*}, \quad \text{ with } \quad\bigcup_{n}\cF_n=\H_{G,Y}^{*},
	\]
	such that each quotient $\cF_{n+1}/\cF_{n}$ is isomorphic to a lowest weight Verma module $M^{-}(\lambda_n)$ for some $\lambda_n\in k$.
	In particular the integer sequence  $\left(\left| G \backslash \binom{Y}{n} \right|\right)_{n\in \mN}$ decomposes into a sum of constant $1,1,1,1,\ldots$-sequences. 
\end{theo}
\begin{proof}
	The crucial step is that the map $e\cdot \blank$ on $(\C^{2})_{\operatorname{fin}}^{\otimes \N}$ is injective. For this we view $(\C^{2})_{\operatorname{fin}}^{\otimes \N}$ as subspace
	\[
	(\C^{2})_{\operatorname{fin}}^{\otimes \N}=\bigoplus_{n\in \mN} \cC\left(\binom{Y}{n}\right)\subseteq \bigoplus_{n\in \mN} \Maps\left(\binom{Y}{n}, k\right). 
	\]
	In \cite[\S5]{cameron90} Cameron turns $\bigoplus_{n\in \mN} \Maps\left(\binom{Y}{n}, k\right)$ into a graded algebra called the incidence algebra. Taking the product with the constant $1$-function on $Y=\binom{Y}{1}$ defines the linear map
	\begin{align*}
		\Maps\left(\binom{Y}{n}, k\right)&\rightarrow \Maps\left(\binom{Y}{n+1}, k\right), \quad v \mapsto [T \mapsto \sum_{S\subseteq T, \, |S|=n}\! \! \!v(S)].
	\end{align*}
	This map is injective provided $|Y|\geq 2n+1$, see \cite[\S5]{cameron90}. 
	By comparing definitions it restricts to our action $e\cdot \blank\colon \cC\left(\binom{Y}{n}\right)\to \cC\left(\binom{Y}{n+1}\right)$, which is hence injective, which also implies that $e\cdot \blank\colon H^{n}_{G,Y}\to H^{n+1}_{G,Y}$ is injective by taking invariants.
	
	Now we can construct a filtration on $\H_{G,Y}^{*}$ inductively. 
	Define $\cF_0\coloneqq 0$. Assume $\cF_0, \ldots, \cF_n$ were constructed for some $n\in \mN$. Consider the quotient $\H_{G,Y}^{*}/ \cF_n$. First note that it decomposes into eigenspaces for $h\cdot \blank $ given by $\H_{G,Y}^{m}/ (\H_{G,Y}^{m} \cap \cF_n)$ with eigenvalue $-\mu(Y)+2n$. Since the action of $f\in \sl_2(\C)$ on $\H_{G,Y}^{*}$ is nilpotent, it is nilpotent when considered as endomorphism of the quotient. Choose $0\neq v\in \H_{G,Y}^{m}/ (\H_{G,Y}^{m}\cap \cF_n)$ such that $f\cdot v=0$ with $m\in \mN$ minimal. The action $e\cdot \blank$ on the quotient $\H_{G,Y}^{*}/ \cF_n$ is injective. Indeed we saw that $e\cdot \blank\colon \H_{G,Y}^{*}\to \H_{G,Y}^{*}$. Moreover is a general fact about vector spaces, that given any monomorphism $\varphi\colon V\hookrightarrow W$ between vector spaces $V$ and $W$ and any subspace $U\subseteq V$ the induced map $\bar{\varphi}\colon V/U\to W/f(U)$ remains injective. Hence $v$ generates a submodule $\bar{\cF_{n+1}}\subseteq \H_{G,Y}^{*}/ \cF_n$ isomorphic to a lowest weight Verma module $M^{-}(\lambda_n)$ for some $\lambda_n\in k$. Now lift $\bar{\cF}_{n+1}$ to a submodule $\cF_{n+1}\subseteq \H_{G,Y}^{*}$ containing $\cF_{n}$. Since each $h$-eigenspace $H^{n}_{G,Y}$ is finite dimensional (its dimension is $\left|G\backslash\binom{Y}{n}\right|$, which is finite since $Y$ is smooth and orbit-finite), every element in $\H_{G,Y}^{*}$ is contained in some $\cF_{m}$ for $m>>0$.   
	%Consider the quotient $\H_{G,Y}^{*}/\cF_n$. $\cF_1\subseteq \H_{G,Y}^{*}$ to be the $\sl_2(\C)$-subrepresentation generated by $v_{\emptyset}$. Since $e\cdot \blank$ is injective, $f\cdot v_{\emptyset}=0$, and $h\cdot v_{\emptyset}=-\mu(Y)v_{\emptyset}$ this representation is isomorphic to the lowest weight Verma module $M^{-}(-\mu(Y))$. Consider the quotient $\H_{G,Y}^{*}/\mathcal{F_1}$. The induced action $e\cdot \blank$ on this quotient remains injective. Indeed this is a general fact about vector spaces, given any monomorphism $\varphi\colon V\hookrightarrow W$ between vector spaces $V$ and $W$ and any subspace $U\subseteq V$ the induced map $\bar{\varphi}\colon V/U\to W/f(U)$ remains injective.
%	
%	The subspace $\H_{G,Y}^{n}=\cC\left( \binom{Y}{n} \right)^G$ is the eigenspace for the action of $h$ with weight (i.e.\ eigenvalue) $2n-\mu(Y)$. We consider the usual partial order $\leq$ on weights, where $\lambda_2\leq \lambda_1$ if $\lambda_1-\lambda_2\in 2\cdot \mN$, for $\lambda_1, \lambda_2 \in k$. In particular the minimal weight occurring is $-\mu(Y)$.
%	The action of $f$ takes a weight vector of weight $\lambda$, and creates a weight vector of weight $\lambda-2$, in particular $f\cdot \blank$ is nilpotent. We choose a  weight vector $v_0$ in the kernel of $f$, it generates the lowest weight Verma module $\cF_1=M^{-}(-\mu(Y))$, since the action of $e$ is injective. 
%	Next consider $\H_{G,Y}^{*}/\mathcal{F_1}$.
%	We proceed recursively by choosing a new vector $v_{r+1}$ in the kernel of $f\cdot \blank$ not contained in the sum of the previously chosen Verma modules.
\end{proof}
%\subsection{The $\sl_2$-Action and Stanley's Argument} %\label{subsection: sl2 and Stanley}
%We first explain Lemma \ref{lemma action on finite stuff} further in the case, where the group $G$ is the automorphism group of a homogeneous oligomorphic graph.
%\begin{prop} \label{proposition for homogenous graphs}
%	Let $\cS$ be a homogeneous oligomorphic graph with vertex set $X$. Consider the group $G=\Aut(\cS)$. The following facts hold about the action $\sl_2(k)\acts (k^2)_{\operatorname{fin}}^{\otimes X}$:
%	\begin{enumerate}
%		\item The space $(k^2)_{\operatorname{fin}}^{\otimes X}$ is the direct sum $\bigoplus_{n\in \mN} \cC(\binom{X}{n})$ considered as $G$-representation.
%		\item The subspace $\left(\cC(\binom{X}{n})\right)^{G}$ has a basis $v_{[\cs]}$ labeled by isomorphism classes $[\cs]$ of $n$-element subgraphs $\cs\subseteq \cS$.  
%		\item We have 
%		\[
%		e\cdot v_{[\cs]}=\!\sum_{[\cs']\in G\backslash \binom{X}{n+1}} \!\left| {\cs'}^{[\cs]} \right|\cdot  v_{[\cs']}, \quad f\cdot v_{[\cs]}=\!\sum_{[\cs']\in G\backslash \binom{X}{n-1}}\!\mu\left({(\cS\setminus \cs')}^{[\cs]} \right)\cdot  v_{[\cs']}.
%		\]
%	\end{enumerate}
%\end{prop}
%
%\subsection{Example I: The infinite symmetric group}
%We set $k=\C$,  $G=S_{\infty}$ and consider $Y=X=\N$. We fix $\lambda\in \C$ and consider the measure $\mu_{\lambda}$ on $S_{\infty}$ defined in Example \ref{example measure for S}.

\begin{beis} \label{example sl3}
	Let $r\in \N$. Then $\left((\C^r)_{\operatorname{fin}}^{\otimes \N}\right)^{S_{\infty}}$ is isomorphic as a vector space to a polynomial ring $\C[x_1,\ldots, x_{r-1}]$ in $(r-1)$-many variables.
	For $a_1,\ldots, a_{r-1}\in \mN$ the basis vector $x_1^{a_1}x_2^{a_2}\cdots x_{r-1}^{a_r}$ corresponds to the $S_{\infty}$-orbit on $\pa_r^{\operatorname{fin}}(\N)$ of the element $(\{1,\ldots, a_1\}, \{a_1+1,\ldots, a_2\}, \ldots, \{\sum_{i=1}^{r-2}a_i+1, \ldots, \sum_{i=1}^{r-1}a_i, \N\setminus \{1,\ldots, \sum_{i=1}^{r-1}a_i\}\})$.
	For instance, consider $r=3$ and write $x=x_1$ and $y=x_2$. Then 
	%The action of $\gl_r(k)$ is given by the formulas
	\[
					\left((\C^{3})_{\operatorname{fin}}^{\otimes \N}\right)^{G} = 	\cbox{\begin{tikzpicture}
							
							% --- Parameters ---
							% Simple roots for sl3 with 60° angle
							\pgfmathsetmacro{\sq}{sqrt(3)/2}
							\coordinate (alpha) at (1,0);  % horizontal direction
							\coordinate (beta) at (0.5,\sq);  % 60° up direction
							\coordinate (gamma) at (-0.5,\sq); % -120° direction (completes the triangle)
							
							% --- 1. Base point -3*omega1 ---
							\pgfmathsetmacro{\basex}{-2*1 + -1*0.5}
							\pgfmathsetmacro{\basey}{-2*0 + -1*\sq}
							\coordinate (Base) at (\basex,\basey);
							
							% --- 2. Create the foreground Pascal triangle coordinates ---
							% Level 1:
							\coordinate (X) at ($(Base)+(alpha)$);
							\coordinate (Y) at ($(Base)+(beta)$);
							
							% Level 2:
							\coordinate (XX) at ($(X)+(alpha)$);
							\coordinate (XY) at ($(X)+(beta)$);
							\coordinate (YY) at ($(Y)+(beta)$);
							
							% Level 3:
							\coordinate (XXX) at ($(XX)+(alpha)$);
							\coordinate (XXY) at ($(XX)+(beta)$);
							\coordinate (XYY) at ($(XY)+(beta)$);
							\coordinate (YYY) at ($(YY)+(beta)$);
							
							% Level 4:
							\coordinate (XXXX) at ($(XXX)+(alpha)$);
							\coordinate (XXXY) at ($(XXX)+(beta)$);
							\coordinate (XXYY) at ($(XXY)+(beta)$);
							\coordinate (XYYY) at ($(XYY)+(beta)$);
							\coordinate (YYYY) at ($(YYY)+(beta)$);
							
							% Level 5 (dots):
							\coordinate (XXXXX) at ($(XXXX)+(alpha)$);
							\coordinate (XXXXY) at ($(XXXX)+(beta)$);
							\coordinate (XXXYY) at ($(XXXY)+(beta)$);
							\coordinate (XXYYY) at ($(XXYY)+(beta)$);
							\coordinate (XYYYY) at ($(XYYY)+(beta)$);
							\coordinate (YYYYY) at ($(YYYY)+(beta)$);
							
							% --- 3. Draw background lattice - VERY tightly cropped ---
							% Only draw lattice points that are within the convex hull of the triangle plus a tiny margin
							\foreach \i in {-2,...,4} {
								\foreach \j in {-2,...,4} {
									\pgfmathsetmacro{\px}{\i*1 + \j*0.5}
									\pgfmathsetmacro{\py}{\i*0 + \j*\sq}
									
									% Only draw if point is very close to our triangle
									% Triangle roughly spans from x=-2.5 to x=3.5, y=-0.9 to y=2.6
									\pgfmathparse{(\px>-2.8) && (\px<3.8) && (\py>-1.2) && (\py<2.9)}
									\ifnum\pgfmathresult=1
									\coordinate (P) at (\px,\py);
									
									% Draw very short line segments
									\draw[gray!25, very thin] ($(P)-0.5*(alpha)$) -- ($(P)+0.5*(alpha)$);
									\draw[gray!25, very thin] ($(P)-0.5*(beta)$) -- ($(P)+0.5*(beta)$);
									\draw[gray!25, very thin] ($(P)-0.5*(gamma)$) -- ($(P)+0.5*(gamma)$);
									\fi
								}
							}
							
							% --- 4. Place the "1" node ---
							\node[font=\normalsize] at (Base) {1};
							
							% --- 5. Place all monomial nodes ---
							\node[font=\normalsize] at (X) {$y$};
							\node[font=\normalsize] at (Y) {$x$};
							
							\node[font=\normalsize] at (XX) {$y^2$};
							\node[font=\normalsize] at (XY) {$xy$};
							\node[font=\normalsize] at (YY) {$x^2$};
							
							\node[font=\normalsize] at (XXX) {$y^3$};
							\node[font=\normalsize] at (XXY) {$xy^{2}$};
							\node[font=\normalsize] at (XYY) {$x^{2}y$};
							\node[font=\normalsize] at (YYY) {$x^3$};
							
							\node[font=\normalsize] at (XXXX) {$y^4$};
							\node[font=\normalsize] at (XXXY) {$xy^3$};
							\node[font=\normalsize] at (XXYY) {$x^2y^2$};
							\node[font=\normalsize] at (XYYY) {$x^3y$};
							\node[font=\normalsize] at (YYYY) {$x^4$};
							
							% Small dots to show continuation
							\foreach \p in {XXXXX, XXXXY, XXXYY, XXYYY, XYYYY, YYYYY}
							\node[circle,fill=black,inner sep=0.6pt] at (\p) {};
						\end{tikzpicture}}
	\]
	
	In this picture the action of the Lie algebra $\gl_3(\C)$ creates/deletes variables according to the gray lines in the background. 
	%We encourage the reader to compute the $\gl_3(\C)$-action here. 
\end{beis}
\begin{rema}[Interpolation 2]
	Really one should think that the basis in Example \ref{example sl3} is not the basis of a polynomial ring in $r-1$ variables, but rather the basis of the `degree $\infty$'-part of a polynomial ring in $r$ variables, where the individual degrees of $x_1,\ldots, x_{r-1}$ are finite and the degree of $x_r$ is infinite. The representation $\left((\C^r)_{\operatorname{fin}}^{\otimes \N}\right)^{S_{\infty}}$ of $\gl_r(\C)$ should be thought of as an infinite symmetric power $\Sym^{\infty} \C^r$, i.e.\ some formal limit of $\Sym^l \C^r=L(l \cdot \omega_1)$ for $l\in \mN$, the irreducible representation of highest weight $l\cdot \omega_1$. This module is a lowest weight parabolic Verma module $M^{\mathfrak{p}, -}(-\lambda \omega_{n-1})=U(\gl_r)\otimes_{U(\mathfrak{p})} k_{-\lambda\omega_{n-1}}$. Here $\mathfrak{p}$ is the maximal parabolic Lie subalgebra given by block-lower-triangular matrices of the form $\mathfrak{p}=\begin{psmallmatrix}
		* & * & \cdots & * & 0 \\
		* & * & \cdots & * & 0 \\
		 & & \vdots & & 0 \\
		* & * & \cdots & * & 0 \\ 
		* & * & \cdots & * & * \\ 
	\end{psmallmatrix}\subseteq \gl_r(\C)$. As indicated before these are the parabolic Verma modules which are considered in \cite{aizenbud2015} (up to the subtlety that we consider lowest weight modules).
\end{rema}

 \subsection{Examples II: Disjoint unions and tensor products} \label{subsection disjoint unions}
  In this section we explain how taking tensor products interacts with the construction in this paper. We explain how the combinatorial model in Subsection \ref{subsection graphs} connects to Lemma \ref{lemma disjoint unions is oligomorphic} on products of oligomorphic groups, and how this connects to tensor products of Lie algebra representations.
 \begin{defi} \label{definition exterior direct sum}
 	Let $I_1$, $I_2$ be indexing sets. The \textit{(exterior) disjoint union} of a $I_1$-multi-relational graph $\cS_1$ and a $I_2$-multi-relational graph $\cS_2$ is the $(\{1,2\}\sqcup I_1\sqcup I_2)$-multi-relational graph $\cS_1\sqcup \cS_2$ with vertices $X_1\sqcup X_2$ and edges
 	$E_1=\{(x_1,x_1)\mid x_1\in X_1\}$, $E_2=\{(x_2,x_2)\mid x_2\in X_2\}$, $E_{i}\coloneqq E_{i}$ for $i\in I_1\sqcup I_2$.
 \end{defi}
 \begin{lemm} 
 	Assume we are in the setting of Definition \ref{definition exterior direct sum}. We have $\Aut(\cS_1\sqcup \cS_2)=\Aut(\cS_1)\times \Aut(\cS_2)$. If $\cS_1$ and $\cS_2$ are homogeneous oligomorphic then $\cS_1\sqcup \cS_2$ also is homogeneous oligomorphic. Finite subgraphs of $\cS_1\times \cS_2$ consists of pairs $(\cs_1, \cs_2)$ of subgraphs $\cs_1\subseteq \cS_1$, $\cs_2\subseteq \cS_2$. 
 \end{lemm}
 \begin{proof}
 	The first and third assertion follow from Definition \ref{definition exterior direct sum}, the second assertion follows from Lemma \ref{lemma disjoint unions is oligomorphic}.
 \end{proof}
% \begin{beis}
% 	Fix $n\in \mN$. Consider the finite symmetric group $G_1=S_n=\Sym(\{1,\ldots, n\})$ and an infinite symmetric group $G_2=S_{\infty}=\Sym(\N)$. 
% \end{beis}
 \begin{beis} \label{defi r-colored finite sets}
 	Let $m\in \N$. We set $I=\{1,\ldots, r\}$ and consider the $I$-multi-relational graph $\N^{\sqcup r}\coloneqq \bigsqcup_{i=1}^{m}\N$, where $\N$ is viewed as $\emptyset$-multi-relational graph. We write the vertices $X_r$  $=\N\times \{1,\ldots, m\}$ and for $i\in \{1,\ldots, m\}$ the set of $i$-edges is $E_i=\{((n,i),(n,i)) \mid n,m\in \N \}$. One can imagine $\N^{\sqcup m}$ as $m$ differently color copies of $\N$. The shared color of two elements indicates that they have $i$-edges to themselves:
 	
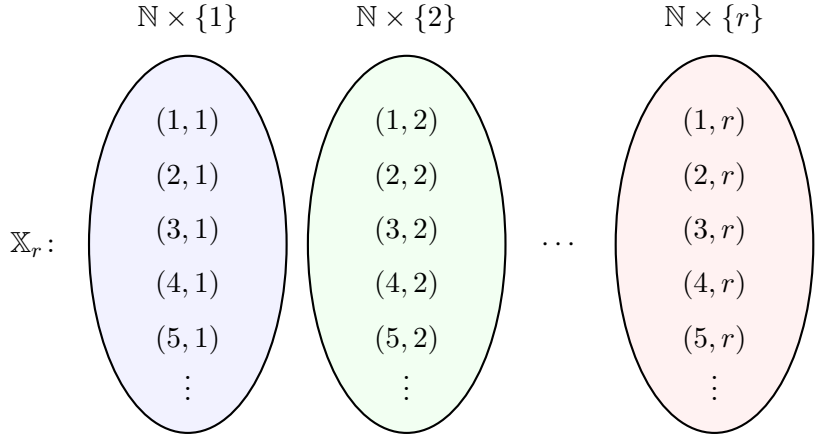
\begin{center}
 		\begin{tikzpicture}[
 			every node/.style={font=\small},
 			set/.style={
 				ellipse,
 				draw,
 				thick,
 				minimum width=2.6cm,
 				minimum height=5cm
 			}, scale=0.9
 			]
 			
 			% Abstand 1 -> 2 (kleiner)
 			\def\distA{3.2}
 			
 			% Abstand 2 -> 3 (groß, mit cdots)
 			\def\distB{4.5}
 			
 			% Beschriftung links vom Bild
 			\node at (-2.2,0) {$\cS_r\colon$};
 			
 			% -------- Erste Ellipse --------
 			\node[set, fill=blue!5] (A) at (0,0) {};
 			\node at (A.north) [above=4pt] {$\mathbb{N}\times\{1\}$};
 			
 			\foreach \y/\n in {1.8/1, 1.0/2, 0.2/3, -0.6/4, -1.4/5}
 			\node at ($(A.center)+(0,\y)$) {$({\n},1)$};
 			\node at ($(A.center)+(0,-2.0)$) {$\vdots$};
 			
 			% -------- Zweite Ellipse --------
 			\node[set, fill=green!5] (B) at (\distA,0) {};
 			\node at (B.north) [above=4pt] {$\mathbb{N}\times\{2\}$};
 			
 			\foreach \y/\n in {1.8/1, 1.0/2, 0.2/3, -0.6/4, -1.4/5}
 			\node at ($(B.center)+(0,\y)$) {$({\n},2)$};
 			\node at ($(B.center)+(0,-2.0)$) {$\vdots$};
 			
 			% -------- Punkte --------
 			\node at ({\distA + 0.5*\distB},0) {$\cdots$};
 			
 			% -------- Dritte Ellipse --------
 			\node[set, fill=red!5] (C) at ({\distA + \distB},0) {};
 			\node at (C.north) [above=4pt] {$\mathbb{N}\times\{r\}$};
 			
 			\foreach \y/\n in {1.8/1, 1.0/2, 0.2/3, -0.6/4, -1.4/5}
 			\node at ($(C.center)+(0,\y)$) {$({\n},r)$};
 			\node at ($(C.center)+(0,-2.0)$) {$\vdots$};
 		\end{tikzpicture}
 	\end{center}
 	The age of $\N^{\sqcup m}$ is the class of \textit{$m$-colored finite sets}. The number of $n$-element subgraphs of $\N^{\sqcup m}$ is  $\left|S_{\infty}^r\backslash\binom{\N\times\{1,\ldots, m\}}{n}\right|=\binom{n+m-1}{n}$, cf.\ Example \ref{example: polynomial ring generating function}.
 	%They can be described as pairs $\cs=(s, c\colon s\to \{1,\ldots, r\})$ of a finite set $s$ together with a coloring map $c$ of the elements of $s$ in $r$ colors.
 \end{beis}
 \begin{lemm} \label{lemma sum measure}
 	Let $\cS_1$, $\cS_2$ be homogeneous oligomorphic graphs. Let $G_1\subseteq \Sym(X_1)$ and $G_2\subseteq \Sym(X_2)$ be two oligomorphic permutation groups. Let $\nu_1$, $\nu_2$ be R-measures on $\cS_1$, $\cS_2$ with values in $k$. 
 	
 	\begin{enumerate}
 		\item \label{item 1 tensor products} There is a disjoint union measure $\mu_1\sqcup \mu_2$ on $\cS_1\sqcup \cS_2$ with values in $k$, which assigns to a pair $(\cs_1, \cs_2)$, where $\cs_1\subseteq \cS_1$, $\cs_2\subseteq \cS_2$ the value $\nu(\cs_1)+\nu(\cs_2)$.
 		\item \label{item 2 tensor products} 
 		Let $G_1=\Aut(\cS_1)$, $G_2=\Aut(\cS_2)$. As $\sl_2(k)$-representations we have an isomorphism
 		\[
 			\H_{G_1\times G_2, X_1\sqcup X_2}^{\star}\cong \H_{G_1, X_1}^{\star} \otimes_k \H_{G_2, X_2}^{\star},
 		\]
 		where the first action of $\sl_2(k)$-action uses the measure $\mu_{\nu_1\sqcup \nu_2}$ on $G_1\times G_2$, the second uses the measure $\mu_{\nu_1}$ on $G_1$, and the third the measure $\mu_{\nu_2}$ on $G_2$. 
 	\end{enumerate}
 \end{lemm}
 \begin{proof}
 	Part \ref{item 1 tensor products} follows, since amalgamations for subgraphs of $\cS_1$ and $\cS_2$ don't interact with each other. For \ref{item 2 tensor products} one uses the bijection from Lemma \ref{lemma disjoint unions is oligomorphic}. It gives an explicit $1\colon 1$-correspondence on bases and hence a vector space isomorphism. A direct computation shows that the $\sl_2(\C)$-actions agree.
 \end{proof}
% \begin{proof}
% 	The open subgroups in $G_1\times G_2$ are exactly of the form $U_1\times U_2$ where $U_1\subseteq G_1$, $U_2\subseteq G_2$ are open subgroups.
% \end{proof}
 \begin{beis}
 	Let $m\in \mN$. Consider the complete graph $K_m$, whose vertices are $\{1,\ldots, m\}$ and edges are $E=\{1,\ldots, m\}^2$. We consider the graph $K_m\sqcup \N$ as in Definition \ref{definition exterior direct sum} with vertex set $\{1,\ldots, m\} \sqcup \N$.
 	By Lemma \ref{lemma sum measure} the unique measure on $S_m=\Aut(K_m)$ (which counts embeddings, see Remark \ref{intuition for measure}) and the measure $\mu_{\lambda}$ for $\lambda\in k$ on $S_{\infty}$ combine to a measure on $S_m\times S_{\infty}$. Lemma \ref{lemma sum measure} gives us an isomorphism
 	\[ 
 	\H_{S_m\times S_{\infty}, \{1,\ldots, m\}\sqcup \N}^{\star}\cong \H_{S_{m}, \{1,\ldots, m\}}^{\star} \otimes_\C \H_{S_{\infty}, \N}^{\star}\cong L(m)\otimes_k M^{-}(-\lambda)
 	\] of $\sl_2(\C)$-representations. Here $L(m)=((\C^{2})^{\otimes m})^{S_m}$ is the $(m+1)$-dimensional irreducible representation of $\sl_2(\C)$, corresponding to the finite integer sequence $1,1,\ldots, 1$ of $(m+1)$-many $1$'s. The representation $L(m)\otimes_k M^{-}(-\lambda)$ has a finite filtration with subquotients
 	\[
 	M^{-}(-\lambda-m), \quad  M^{-}(-\lambda-m+2), \quad  \cdots, \quad  M^{-}(-\lambda+m).
 	\]
 	This corresponds to the sequence $1,2,\dots, m,m+1,m+1,m+1,\ldots$ decomposing into a sum of $(m+1)$ shifted sequences $1,1,1,1,\ldots$.
 \end{beis}
 \begin{beis}
 	Let $\lambda_1, \lambda_2\in \C$. Consider the corresponding R-measures $\nu_{\lambda_1}$ on $\N$ and $\nu_{\lambda_2}$ on $\N$. Then 
 	 \[
 	 \H_{S_{\infty}\times S_{\infty}, \N\sqcup \N}^{\star}\cong \H_{S_{\infty}, \N}^{\star} \otimes_\C \H_{S_{\infty}, \N}^{\star} \cong M^{-}(-\lambda_1)\otimes_\C M^{-}(-\lambda_1).
 	 \]
 	 This representation has an infinite filtration with subquotients
 	 \[
 	 	M^{-}(-\lambda_1-\lambda_2), \quad M^{-}(-\lambda_1-\lambda_2+2), \quad M^{-}(-\lambda_1-\lambda_2+4), \quad \cdots
 	 \]
 	 see \cite[\S4.1]{merceron2025} for more on this.
 	 The integer sequence in this case is $1,2,3,4,\ldots$, which decomposes as a sum of constant $1$-sequences.
 \end{beis}
 \subsection{Example III: Fibonacci numbers} \label{section Fibonacci}
 Throughout this section we fix a set $I$. Moreover we fix a homogeneous oligomorphic $I$-multi-relational graph $\cS$ with vertex set $X$.
%	
%	We the complete directed graph $K_m$ with vertices $\{1,\ldots, m\}$ and edges $E=\{1,\ldots, m\}^2$. We will consider a $\{1,2\}$-multi-relational graph built out of $K_m$ and $\mQ$. 
 \begin{defi}
 	 We consider the $(I\sqcup \{2\})$-multi-relational graph $\cS\times \mQ$ whose vertex set is $X\times \mQ$ and whose edge sets are
 	\begin{gather*}
 		\tilde{E}_i=\{ ((x_1,y),(x_2,y)) \mid (x_1,x_2)\in E_i\}\subseteq (X\times \mQ)^{2} \quad \text{for $i\in I$,}\\
 		E_{2}=\{(x_1,y_1),(x_2,y_2) \mid x_1,x_2\in X, \{y_1<y_2\} \subseteq \mQ\}.
 	\end{gather*}
 \end{defi}
 \begin{beis}
 	Consider the complete graph $\cS=K_2$. It consists of two vertices connected by a blue edges in both directions, which we image as one blue undirected edge.
 	The graph $K_2\times \mQ$ is an infinitely squeezed together sponge, built locally out of such pieces:
 	\[
% 	\begin{tikzpicture}[
% 		vertex/.style={circle,fill,inner sep=2pt},
% 		inneredge/.style={blue!60!black, line width=1.6pt},
% 		outeredge/.style={red!60!black, -{Latex[length=2mm]}, line width=0.8pt}
% 		]
% 		
% 		% Parameters
% 		\def\sep{2.2}      % horizontal spacing between fibers
% 		\def\v{0.9}        % vertical separation inside each K2
% 		\def\levels{-1,0,1,2}
% 		
% 		% Draw fibers (copies of K2)
% 		\foreach \i in \levels {
% 			
% 			% Vertices of K2
% 			\coordinate (T\i) at (\i*\sep, \v);
% 			\coordinate (B\i) at (\i*\sep,-\v);
% 			
% 			% Thick internal K2 edge
% 			\draw[inneredge] (T\i) -- (B\i);
% 			
% 			% Vertices
% 			\node[vertex] at (T\i) {};
% 			\node[vertex] at (B\i) {};
% 		}
% 		
% 		% Complete bipartite connections between consecutive fibers
% 		\foreach \i in {-1,0,1}{
% 			\pgfmathtruncatemacro{\next}{\i+1}
% 			\foreach \x in {T,B}{
% 				\foreach \y in {T,B}{
% 					\draw[outeredge] (\x\i) -- (\y\next);
% 				}
% 			}
% 		}
% 		
% 		% Continuation dots
% 		\node at (-3.2,0) {$\cdots$};
% 		\node at (6.6,0) {$\! \! \! \! \cdots$};
% 		
% 		% Labels
% 		\node at (-4.3,0) {$K_2\times \mQ: \qquad $};
% 		%\node[above] at (0, \v+0.6) {$\{1\}\times\mathbb{Q}$};
% 		%\node[below] at (0,-\v-0.6) {$\{-1\}\times\mathbb{Q}$};
% 		
% 	\end{tikzpicture}
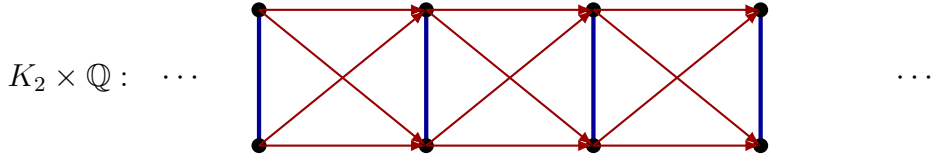
\begin{tikzpicture}[
	vertex/.style={circle,fill,inner sep=2.2pt},
	inneredge/.style={green!80!black, line width=2pt},
	outeredge/.style={
		orange,
		line width=1.2pt,
		-{Latex[length=4mm, width=2.5mm]}
	}
	]
	
	% Parameters
	\def\sep{2.6}      % horizontal spacing between fibers
	\def\v{1.0}        % vertical separation inside each K2
	\def\levels{-1,0,1,2}
	
	% Draw fibers (copies of K2)
	\foreach \i in \levels {
		
		% Vertices
		\coordinate (T\i) at (\i*\sep, \v);
		\coordinate (B\i) at (\i*\sep,-\v);
		
		% Internal edge
		\draw[inneredge] (T\i) -- (B\i);
		
		% Nodes
		\node[vertex] at (T\i) {};
		\node[vertex] at (B\i) {};
	}
	
	% Bipartite connections
	\foreach \i in {-1,0,1}{
		\pgfmathtruncatemacro{\next}{\i+1}
		\foreach \x in {T,B}{
			\foreach \y in {T,B}{
				\draw[outeredge] (\x\i) -- (\y\next);
			}
		}
	}
	
	% Continuation dots
	\node at (-3.8,0) {$\cdots$};
	\node at (7.6,0) {$\cdots$};
	
	% Label
	\node at (-5.2,0) {$K_2 \times \mathbb{Q}:\quad$};
	
\end{tikzpicture}
 	\]
 	Here the \textcolor{green!80!black}{green unoriented edges} represent the $1$-edges (which exist in both directions, hence unoriented), while the \textcolor{orange}{orange directed edges} symbolize the $2$-edges coming from the total order on $\mQ$.
 \end{beis}
 We gather some elementary properties about $\cS\times \mQ$ in a proposition.
 \begin{prop}[Folklore] \label{proposition properties of some wreath products}
 	The graph $\cS\times \mQ$ satisfies the following properties:
 	\begin{enumerate}
 		\item Its automorphism group is isomorphic to the wreath product $\Aut(\cS) \wr \Aut(\mQ, <)$. 
 		\item It is homogeneous oligomorphic.
 		\item \label{number composition with certain summands} The isomorphism classes of subgraphs of $\cS\times \mQ$ correspond to $l$-tuples $([\cs_1],\ldots, [\cs_l])$ of any length $l\in \mN$ whose entries are isomorphism classes $[\cs_j]$ of finite non-empty subgraphs $\cs_j\subseteq \cS$, where $1\leq j\leq l$. The $2$-edges in $([\cs_1],\ldots, [\cs_l])$ declare $\cs_1<\cs_2<\ldots <\cs_l$.
 	\end{enumerate}
 \end{prop}
 \begin{proof}
 	All properties are clear by inspection.
 \end{proof}
 \begin{beis} \label{example Fibonacci numbers}
 	The complete graph $K_2$ has two finite non-empty subgraphs, corresponding to numbers $1$ and $2$.
 	For $n\in \mN$ the isomorphism-classes of $n$-element subgraphs of $K_2\times \mQ$ correspond to tuples of $1$'s and $2$'s, whose entries sum up to $n$. The number of such tuples is exactly the $n$-th Fibonacci number with the usual conventions $F_0=1$, $F_1=1$, and $F_{n}=F_{n-1}+F_{n-2}$ for $n\geq 2$. 
 	Similarly one obtains for $K_3\times \mQ$ the Tribonacci numbers $1, 1, 2, 4, 7, 13,24,\ldots$, for $K_4\times \mQ$ the Tetranacci numbers
 	$1, 1, 2, 4, 8, 15, 29$, and so on.
 \end{beis}
 
 \begin{theo} \label{theorem measure on Fibonacci}
 	Let $k$ be a field. Let $\nu$ be an R-measure on $\age(\cS)$ with values in $k$. Then there is an R-measure $\nu_{\mQ}$ on $\cS\times \mQ$, which assigns to the graph $(\cs_1,\cs_2,\ldots, \cs_l)$, cf.\ Proposition \ref{proposition properties of some wreath products}, the value obtained as a product $
 		\nu_{\mQ}(\cs_1,\cs_2,\ldots, \cs_l)=(-1)^l \prod_{i=1}^{l} \nu(\cs_i)\in k\setminus\{0\}$.
 \end{theo}
 \begin{proof}
 	The only non-trivial part is the multiplicativity of amalgamations condition in Definition \ref{defi R-measure}. The crucial idea of the proof is to use Lemma \ref{technical lemma} about one-point amalgamations. There are three kinds of one-point amalgamations. Either 
	\begin{enumerate}
		\item \label{case 1} both points are added to the same position, which already exists,
		\item \label{case 2} they are added to the same relative position in between or behind the sequence, which does not exist yet,
		\item \label{case 3} they are added to different positions.
	\end{enumerate}
 	In \ref{case 1} the possible amalgamations all come from $\cS$. In \ref{case 2} the possible amalgamations correspond to one-point amalgamations of the empty subgraph $\emptyset\subseteq \cS$. In both cases multiplicativity for $\nu_{\mQ}$ follows from multiplicativity of $\nu$ using the signs $(-1)^l$.
 	Case \ref{case 3} corresponds to adding a vertex at two different existing places, which gives a unique amalgamation.
 \end{proof}
% \begin{rema}
% 	The `limit' of the groups $\Aut(K_m\times \mQ)=S_m\wr \Aut(\mQ, <)$ as $m$ approaches $\infty$ is $S_{\infty}\wr \Aut(\mQ, <)=\Sym(\N\times \mQ)$. The analogue of Proposition \ref{proposition properties of some wreath products} part \ref{number composition with certain summands} is that orbits $(S_{\infty}\wr \Aut(\mQ, <)) \backslash \binom{\N\times \mQ}{n}$ correspond to all compositions of an integer $n\in \mN$ without restrictions. The number of those is $1$ for $n=0$ and $2^{n-1}$ for $n>0$.
% 	This is not to be confused with \Cref{beis: partitions}, where we considered $\Aut(\mQ, <) \wr S_{\infty}\subseteq \Sym(\mQ\times \N)$, which resulted in the numbers of integer partitions. We will consider a measure for $\Aut(\mQ, <) \wr S_{\infty}$ in the upcoming work \cite{MeyerWojciechowski2026}. 
% \end{rema}
 \begin{beis} \label{example Verma of fibonacci}
 	We consider the graph $\cS=K_2\times (\mQ,<)$. Let $G=S_2 \wr \Aut(\mQ, <)$. As usual, $\mathsf{H}_{G,X}^{\star}$
 	has a basis labeled by $G$-orbits. By Example \ref{example Fibonacci numbers} these correspond precisely to tuples of $1$'s, which we depict as single dots $\cbox{\seq{1}}$ and $2$'s, which we depict as \hphantom{,} $\cbox{\seq{2}}${\hphantom{,}}. Moreover $\dim_k \mathsf{H}_{G,X}^{n}=F_n$.
 	Consider Cameron's action of $\textcolor{red!60!black}{e}$ on $\mathsf{H}_{G,X}^{\star}$ in figure \ref{the action of e}.
 	The positive integers within each arrow count the number of embeddings between the indicated subgraphs. Note that the sum of all numbers going into a fixed finite graph $\cs$ is the number of its vertices $n=|\cs|$.
 	We consider the R-measure on $K_2\times \mQ$ induced by the unique measure on $K_2$ via Theorem \ref{theorem measure on Fibonacci}. 
 	It assigns to the one-point graph $\seq{1}$ the value $-2$, since $-2=(-1)\cdot 2$ and $2$ is the number of embeddings of the one-vertex subgraph into $K_2$. Moreover the the set of vertices $X$ is precisely the set of embeddings $X^{[\seq{1}]}=\{\seq{1} \hookrightarrow X\}$. This implies $\mu(X)=-2$. Hence the action of $h=\begin{psmallmatrix}
 		1 & 0 \\
 		0 & -1
 	\end{psmallmatrix}$ acts on $\cC\left(\binom{X}{n}\right)$ by the scalar $2n-(-2)=2n+2$. We get the following action of
 	$\textcolor{blue!60!black}{f}=\begin{psmallmatrix}
 					0 & 0 \\ 
 					1 & 0
 				\end{psmallmatrix}$:
 			
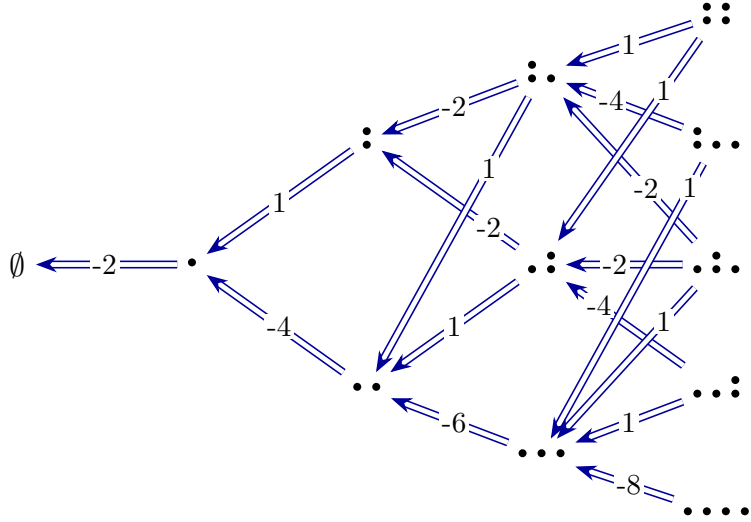
\begin{figure}[ht]
 				\centering
 				
 				\begin{tikzpicture}[
 					element/.style={inner sep=2pt, outer sep=2pt},
 					>=Stealth,
 					arrow/.style={blue!60!black, ->, double=white, double distance=1.8pt, line width=0.6pt},yscale=1.1
 					]
 					
 					% Shift everything to start at x=0.5cm from left edge
 					\begin{scope}[xshift=0.5cm]
 						
 						% Column spacing: 2.3cm
 						% Increased vertical spacing for better number placement
 						
 						% Column 0: n=0 at x=0
 						\node[element] (n0) at (0,0) {$\emptyset$};
 						
 						% Column 1: n=1 at x=2.3
 						\node[element] (n1-1) at (2.3,0) {\seq{1}};
 						
 						% Column 2: n=2 at x=4.6 (increased vertical spacing: 1.5 instead of 1.2)
 						\node[element] (n2-1) at (4.6,1.5) {\seq{2}};
 						\node[element] (n2-2) at (4.6,-1.5) {\seq{1,1}};
 						
 						% Column 3: n=3 at x=6.9 (increased vertical spacing: 2.3 instead of 1.8)
 						\node[element] (n3-1) at (6.9,2.3) {\seq{2,1}};
 						\node[element] (n3-2) at (6.9,0) {\seq{1,2}};
 						\node[element] (n3-3) at (6.9,-2.3) {\seq{1,1,1}};
 						
 						% Column 4: n=4 at x=9.2 (increased vertical spacing: 3.0 instead of 2.4)
 						\node[element] (n4-1) at (9.2,3.0) {\seq{2,2}};
 						\node[element] (n4-2) at (9.2,1.5) {\seq{2,1,1}};
 						\node[element] (n4-3) at (9.2,0) {\seq{1,2,1}};
 						\node[element] (n4-4) at (9.2,-1.5) {\seq{1,1,2}};
 						\node[element] (n4-5) at (9.2,-3.0) {\seq{1,1,1,1}};
 						
 						% Draw all edges from right (larger n) to left (smaller n) - blue arrows
 						
 						% From n=1 to n=0
 						\draw[arrow] (n1-1) -- (n0);
 						
 						% From n=2 to n=1
 						\draw[arrow] (n2-1) -- (n1-1);
 						\draw[arrow] (n2-2) -- (n1-1);
 						
 						% From n=3 to n=2
 						\draw[arrow] (n3-1) -- (n2-1);
 						\draw[arrow] (n3-2) -- (n2-1);
 						\draw[arrow] (n3-1) -- (n2-2);
 						\draw[arrow] (n3-2) -- (n2-2);
 						\draw[arrow] (n3-3) -- (n2-2);
 						
 						% From n=4 to n=3
 						% [2,1,1] and [1,2,1] to [2,1]
 						\draw[arrow] (n4-2) -- (n3-1);
 						\draw[arrow] (n4-3) -- (n3-1);
 						% [2,2] to [2,1]
 						\draw[arrow] (n4-1) -- (n3-1);
 						
 						% [1,2,1] and [1,1,2] to [1,2]
 						\draw[arrow] (n4-3) -- (n3-2);
 						\draw[arrow] (n4-4) -- (n3-2);
 						% [2,2] to [1,2]
 						\draw[arrow] (n4-1) -- (n3-2);
 						
 						% [1,1,1,1] to [1,1,1]
 						\draw[arrow] (n4-5) -- (n3-3);
 						% [2,1,1], [1,2,1], [1,1,2] to [1,1,1]
 						\draw[arrow] (n4-2) -- (n3-3);
 						\draw[arrow] (n4-3) -- (n3-3);
 						\draw[arrow] (n4-4) -- (n3-3);
 						
 						% Add numbers placed directly on top of each arrow with white background
 						% n=1 -> n=0
 						\node[fill=white, inner sep=1pt, font=\small] at ($(n1-1)!0.5!(n0)$) {-2};
 						
 						% n=2 -> n=1
 						\node[fill=white, inner sep=1pt, font=\small] at ($(n2-1)!0.5!(n1-1)$) {1};
 						\node[fill=white, inner sep=1pt, font=\small] at ($(n2-2)!0.5!(n1-1)$) {-4};
 						
 						% n=3 -> n=2
 						\node[fill=white, inner sep=1pt, font=\small] at ($(n3-1)!0.5!(n2-1)$) {-2};
 						\node[fill=white, inner sep=1pt, font=\small] at ($(n3-2)!0.3!(n2-1)$) {-2};
 						\node[fill=white, inner sep=1pt, font=\small] at ($(n3-1)!0.3!(n2-2)$) {1};
 						\node[fill=white, inner sep=1pt, font=\small] at ($(n3-2)!0.5!(n2-2)$) {1};
 						\node[fill=white, inner sep=1pt, font=\small] at ($(n3-3)!0.5!(n2-2)$) {-6};
 						
 						% n=4 -> n=3
 						% [2,1] edges
 						\node[fill=white, inner sep=1pt, font=\small] at ($(n4-2)!0.6!(n3-1)$) {-4};
 						\node[fill=white, inner sep=1pt, font=\small] at ($(n4-3)!0.4!(n3-1)$) {-2};
 						\node[fill=white, inner sep=1pt, font=\small] at ($(n4-1)!0.5!(n3-1)$) {1};
 						
 						% [1,2] edges
 						\node[fill=white, inner sep=1pt, font=\small] at ($(n4-3)!0.58!(n3-2)$) {-2};
 						\node[fill=white, inner sep=1pt, font=\small] at ($(n4-4)!0.67!(n3-2)$) {-4};
 						\node[fill=white, inner sep=1pt, font=\small] at ($(n4-1)!0.3!(n3-2)$) {1};
 						
 						% [1,1,1] edges
 						\node[fill=white, inner sep=1pt, font=\small] at ($(n4-5)!0.5!(n3-3)$) {-8};
 						\node[fill=white, inner sep=1pt, font=\small] at ($(n4-2)!0.15!(n3-3)$) {1};
 						\node[fill=white, inner sep=1pt, font=\small] at ($(n4-3)!0.3!(n3-3)$) {1};
 						\node[fill=white, inner sep=1pt, font=\small] at ($(n4-4)!0.5!(n3-3)$) {1};
 						
 					\end{scope}
 					
 				\end{tikzpicture}
 				
 				\caption{Our action of $f\in \sl_2(\C)$ on $\H_{G,X}^{\star}$ for Fibonacci}
 				\label{fig:mytikz}
 			\end{figure}
 	Here the scalars `count' the number of embeddings between the complements. For instance, consider $\cbox{\seq{2,1}}$, which corresponds to the orbit of $\{(1,0), (2,0), (1,1)\}\in \binom{X}{3}$. We consider the complement of this subset in $X$ and decompose it into $U$-orbits for the open subgroup $U=U((1,0), (2,0), (1,1))$. In this way we get $4$ orbits:
 	
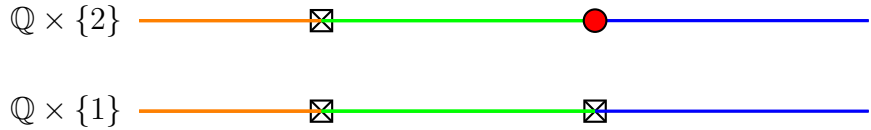
\begin{center}
 		\begin{tikzpicture}[scale=1.2, yscale=-1,
 			cross/.style={draw=black, line width=0.8pt, minimum size=8pt, inner sep=0pt,
 				path picture={
 					\draw (path picture bounding box.south west) -- (path picture bounding box.north east);
 					\draw (path picture bounding box.south east) -- (path picture bounding box.north west);
 			}},
 			dot/.style={circle, fill, inner sep=1.5pt, outer sep=0pt}]
 			
 			% Draw the two horizontal lines
 			\draw[thick] (0,1) -- (8,1);  % top line
 			\draw[thick] (0,0) -- (8,0);  % bottom line
 			
 			% Labels on the left
 			\node at (-0.8,1) {$\mathbb{Q}\times\{1\}$};
 			\node at (-0.8,0) {$\mathbb{Q}\times\{2\}$};
 			
 			% Three excluded points as crosses
 			\node[cross] at (2,1) {};   % top left excluded point
 			\node[cross] at (5,1) {};   % top right excluded point  
 			\node[cross] at (2,0) {};   % bottom excluded point (directly under left top point)
 			
 			% Decompose the complement into 4 pieces:
 			
 			% ORANGE: left of the left excluded points on both lines
 			\draw[orange, thick, line width=1.2pt] (0,1) -- (2,1);
 			\draw[orange, thick, line width=1.2pt] (0,0) -- (2,0);
 			
 			% GREEN: between the left and right excluded points on top line AND directly below on bottom line
 			\draw[green, thick, line width=1.2pt] (2,1) -- (5,1);
 			\draw[green, thick, line width=1.2pt] (2,0) -- (5,0);
 			
 			% RED: the segment on bottom line between the left and right excluded points? No — let's clarify:
 			% Actually: the special point is at (5,0), so the opposite of the green part means:
 			% Green on top is (2,1) to (5,1)
 			% Green on bottom is (2,0) to (5,0) — but that's interrupted by the special point?
 			% Wait, the special point at (5,0) is NOT excluded, so the bottom green segment should stop at (5,0) then continue after?
 			% Let me reconsider: The "opposite side" means: wherever there's a colored segment on top, the corresponding bottom segment (vertically aligned) gets the same color.
 			
 			% So:
 			% ORANGE: left of x=2 on both lines
 			% GREEN: between x=2 and x=5 on both lines
 			% BLUE: right of x=5 on both lines
 			% RED: just the special point at (5,0)
 			
 			\draw[orange, thick, line width=1.2pt] (0,1) -- (2,1);
 			\draw[orange, thick, line width=1.2pt] (0,0) -- (2,0);
 			
 			\draw[green, thick, line width=1.2pt] (2,1) -- (5,1);
 			\draw[green, thick, line width=1.2pt] (2,0) -- (5,0);
 			
 			\draw[blue, thick, line width=1.2pt] (5,1) -- (8,1);
 			\draw[blue, thick, line width=1.2pt] (5,0) -- (8,0);
 			
 			% Mark the special point (already at (5,0) — this is now the endpoint of green on bottom and start of blue)
 			\node[dot, fill=red, draw=black, thick, inner sep=3pt, minimum size=6pt] at (5,0) {};
 			%\node at (5,-0.4) {\small\textbf{special point}};
 			
 		\end{tikzpicture}
 	\end{center}
 	Here they are \textcolor{orange}{orange}, \textcolor{green}{green}, \textcolor{red}{red}, and \textcolor{blue}{blue}. If we add a vertex from \textcolor{green}{green} or \textcolor{blue}{blue} we create graphs isomorphic to $\cbox{\seq{2,1,1}}$ . The coefficient of the action of $f$ is $-4$, since each of the measures of the double lines is $-2$. If we add a vertex from \textcolor{red}{red}, we create $\cbox{\seq{2,2}}$. Since there is only  one vertex, the coefficient is $1$. If we add a vertex from \textcolor{orange}{orange} we create a graph isomorphic to $\cbox{\seq{1,2,1}}$, hence the scalar in the action of $f$ is $-2$.
 \end{beis}
 \subsection*{A question}
 One natural question that remains is the following:
 \begin{center}
 	\textit{When is the rank generating function of an oligomorphic group (in particular a finite permutation groups) log-concave?}
 \end{center}
  This does not hold in general (even for permutation groups of finite sets). There is algebraic geometry flavored theory for matroids, see \cite{adiprasitohuhkatz18}. On the other hand there is work on log-concavity for characters of Verma modules in \cite{kharemathernedizier25}. From the perspective of this paper we see this as hints that there should exist a natural class of oligomorphic permutation groups, which give log-concave sequences coming from an analogue of Hodge theory. How does this class look like? Another hint in this geometric direction is that in \cite{falquetheiry18}, it is shown that if the number of $G$-orbits on $\binom{X}{n}$ is bounded by a polynomial in $n$, then the orbit algebra $\H_{G,X}^{\star}$ is Cohen--Macaulay.

\printbibliography

@misc{R,
	Author = {R. Rouquier},
	Note = {arXiv:0812.5023},
	Title = {2-{K}ac-{M}oody algebras},
	Year = {2008}}

@misc{harmansnowden22,
 author = {Harman, Nate and Snowden, Andrew},
 title = {Oligomorphic groups and tensor categories},
 year = {2022},
 eprint={2204.04526},
archivePrefix={arXiv},
primaryClass={math.RT},
note={Preprint}
}

@unpublished{reinernote,
    author = {Morales, Alejandro H. and Reiner, Victor and Villamizar, Nelly},
    title = {Reflection Groups and Enumeration},
    note = {Lecture notes},
    url = {https://www-users.cse.umn.edu/~reiner/Talks/Reflection_groups_and_enumeration.pdf}
}

@book {cameron90,
    AUTHOR = {Cameron, Peter J.},
     TITLE = {Oligomorphic permutation groups},
    SERIES = {London Mathematical Society Lecture Note Series},
    VOLUME = {152},
 PUBLISHER = {Cambridge University Press, Cambridge},
      YEAR = {1990},
     PAGES = {viii+160},
      %%ISBN = {0-521-38836-8},
   %%MRCLASS = {20B07 (03C35 03C45 03C60 20-02 20B35)},
  %%MRNUMBER = {1066691},
%%MRREVIEWER = {Anand\ Pillay},
       DOI = {10.1017/CBO9780511549809}
}

@article {harmansnowden24delannoy,
    AUTHOR = {Harman, Nate and Snowden, Andrew and Snyder, Noah},
     TITLE = {The {D}elannoy category},
   JOURNAL = {Duke Math. J.},
  FJOURNAL = {Duke Mathematical Journal},
    VOLUME = {173},
      YEAR = {2024},
    NUMBER = {16},
     PAGES = {3219--3291},
      %%ISSN = {0012-7094,1547-7398},
   MRCLASS = {20C99 (18M25)},
  MRNUMBER = {4846194},
MRREVIEWER = {C.\ N\u ast\u asescu},
       DOI = {10.1215/00127094-2024-0012},
}

@misc{merceron2025,
      title={Tensor Products with Verma Module and Restriction to Parabolic Subalgebra}, 
      author={Antoine Merceron},
      year={2025},
      eprint={2509.14220},
      archivePrefix={arXiv},
      primaryClass={math.RT},
      %%url={https://arxiv.org/abs/2509.14220},
      note={Preprint}
}

@article {aizenbud2015,
    AUTHOR = {Aizenbud, Inna Entova},
     TITLE = {Schur-{W}eyl duality for {D}eligne categories},
   JOURNAL = {Int. Math. Res. Not. IMRN},
  FJOURNAL = {International Mathematics Research Notices. IMRN},
      YEAR = {2015},
    NUMBER = {18},
     PAGES = {8959--9060},
      %%ISSN = {1073-7928,1687-0247},
   MRCLASS = {18D10},
  MRNUMBER = {3417700},
MRREVIEWER = {Eric\ C.\ Rowell},
       DOI = {10.1093/imrn/rnu214},
       %%URL = {https://doi.org/10.1093/imrn/rnu214},
}

@article{stanley80a,
 author = {Stanley, Richard P.},
 title = {Weyl groups, the hard {Lefschetz} theorem, and the {Sperner} property},
 fjournal = {SIAM Journal on Algebraic and Discrete Methods},
 journal = {SIAM J. Algebraic Discrete Methods},
 %%issn = {0196-5212},
 volume = {1},
 pages = {168--184},
 year = {1980},
 doi = {10.1137/0601021},
}

@misc{stanley80b,
 author = {Stanley, Richard P.},
 title = {Unimodal sequences arising from {Lie} algebras},
 year = {1980},
 %%language = {English},
 howpublished = {Combinatorics, representation theory and statistical methods in groups, {Young} {Day} {Proc}., {Lect}. {Notes} pure appl. {Math}., {Vol}. 57, 127-136 (1980).},
 %%keywords = {05A10,17B10},
 %%zbMATH = {3704591},
 %%Zbl = {0451.05004}
}

@article{cameron00,
 author = {Cameron, Peter J.},
 title = {Sequences realized by oligomorphic permutation groups},
 fjournal = {Journal of Integer Sequences},
 journal = {J. Integer Seq.},
 %%issn = {1530-7638},
 volume = {3},
 number = {1},
 pages = {art. 00.1.5, no pag.},
 year = {2000},
 %%language = {English},
 %%keywords = {20B07,03E15},
 %%zbMATH = {1522338},
 %%Zbl = {0980.20001}
}

@misc{kharemathernedizier25,
      title={Log-concavity of characters of parabolic {Verma} modules, and of restricted {Kostant} partition functions}, 
      author={Khare, Apoorva and Matherne, Jacob P. and Dizier, Avery St.},
      year={2025},
      eprint={2504.01623},
      archivePrefix={arXiv},
      primaryClass={math.RT},
      note={Preprint}
}

@misc{MeyerWojciechowski2026,
  title        = {On the finite length problem for vector spaces from oligomorphic groups},
  author       = {Meyer, Sebastian and Wojciechowski, Zbigniew},
  note         = {In Preparation}
}

@article{tsankov12,
 author = {Tsankov, Todor},
 title = {Unitary representations of oligomorphic groups},
 fjournal = {Geometric and Functional Analysis. GAFA},
 journal = {Geom. Funct. Anal.},
 %%issn = {1016-443X},
 volume = {22},
 number = {2},
 pages = {528--555},
 year = {2012},
 doi = {10.1007/s00039-012-0156-9},
}

@misc{harmansnowdensnyder23,
author = {Harman, Nate and Snowden, Andrew and Snyder, Noah},
 title = {The circular {Delannoy} category},
 year = {2023},
 eprint={2303.10814},
archivePrefix={arXiv},
primaryClass={math.RT},
note={Preprint}
}

@misc{krizsophietree,
  title={The Delannoy tree category},
  author={Kriz, Sophie},
  url={https://krizsophie.github.io/DelannoyTrees24098.pdf},
  urldate = {2026-15-03},
  note={Preprint}
}

@misc{krizsophiequantum,
  title={Quantum Dellannoy categories},
  author={Kriz, Sophie},
  url={https://krizsophie.github.io/QuantumDelannoyCategory23111.pdf},
  urldate = {2026-15-03},
  note={Preprint}
}

@misc{snowden24,
 author = {Snowden, Andrew},
 title = {The thirty-seven measures on permutations},
 year = {2024},
 eprint={2404.08775},
archivePrefix={arXiv},
primaryClass={math.{CO}},
 note={Preprint}
}

@misc{canrud26,
 author = {Can, Thanh and R{\"u}d, Thomas},
 title = {Measures on {Cameron}'s treelike classes and applications to tensor categories},
 year = {2026},
 eprint={2603.03690},
 archivePrefix={arXiv},
primaryClass={math.{CO}},
 note={Preprint}
}

@misc{harmansnowdeninterpolation,
 author = {Harman, Nate and Snowden, Andrew},
 title = {Classical interpolation categories},
 year = {2025},
 eprint={2507.12216},
 archivePrefix={arXiv},
primaryClass={math.{RT}}
}

@incollection{deligne07,
 author = {Deligne, Pierre},
 title = {The category of representations of the symmetric group {{\(S_t\)}} when {{\(t\)}} is not a natural number.},
 booktitle = {Algebraic groups and homogeneous spaces. Proceedings of the international colloquium, Mumbai, India, January 6--14, 2004},
 %isbn = {978-81-7319-802-1},
 pages = {209--273},
 year = {2007},
 publisher = {New Delhi: Narosa Publishing House/Published for the Tata Institute of Fundamental Research},
 language = {French},
 %%keywords = {20C30,18D10,05E10,20G05,20C08},
 %%zbMATH = {5238760},
 %%Zbl = {1165.20300}
}

@book {mazorchuk10,
    AUTHOR = {Mazorchuk, Volodymyr},
     TITLE = {Lectures on {$\sl_2(\C)$}-modules},
 PUBLISHER = {Imperial College Press, London},
      YEAR = {2010},
     PAGES = {x+263},
      %%ISBN = {978-1-84816-517-5; 1-84816-517-X},
   MRCLASS = {17B10 (16G20 17B20 17B35 17B55)},
  MRNUMBER = {2567743},
MRREVIEWER = {Murray\ R.\ Bremner},
}

@article {falquetheiry18,
    AUTHOR = {Falque, Justine and Thi\'ery, Nicolas M.},
     TITLE = {The orbit algebra of an oligomorphic permutation group with
              polynomial profile is {C}ohen-{M}acaulay},
   JOURNAL = {S\'em. Lothar. Combin.},
  FJOURNAL = {S\'eminaire Lotharingien de Combinatoire},
    VOLUME = {80B},
      YEAR = {2018},
     PAGES = {Art. 83, 12},
      %%ISSN = {1286-4889},
   %%MRCLASS = {20C20 (05E18)},
  %%MRNUMBER = {3940658},
}

@article{adiprasitohuhkatz18,
 author = {Adiprasito, Karim and Huh, June and Katz, Eric},
 title = {Hodge theory for combinatorial geometries},
 fjournal = {Annals of Mathematics. Second Series},
 journal = {Ann. Math. (2)},
 %issn = {0003-486X},
 volume = {188},
 number = {2},
 pages = {381--452},
 year = {2018},
 language = {English},
 doi = {10.4007/annals.2018.188.2.1},
 %keywords = {14T15,05A99,05E16,14F99},
 %zbMATH = {6921184},
 %Zbl = {1442.14194}
}
\end{document}